\documentclass{amsart}  	
\usepackage{geometry}                		
\usepackage{graphicx}
\usepackage{amsmath}
\usepackage[all]{xy}								
\usepackage{amssymb}
\usepackage{amsthm}
\usepackage{fancyhdr}
\usepackage[small,nohug,heads=vee]{diagrams}
\diagramstyle[labelstyle=\scriptstyle]

\numberwithin{equation}{section}
\newtheorem{theorem}{Theorem}[section]

\newtheorem{definition}{Definition}[section]
\newtheorem{lemma}{Lemma}[section]

\newtheorem{remark}{Remark}[theorem]

\newtheorem{claim}{Claim}[theorem]

\title{The K\"ahler-Ricci flow on Fano bundles}
\author{Xin Fu}
\address{Department of Mathematics, Rutgers University, Piscataway, NJ 08854, USA}
\email{xf35@scarlatmail.rutgers.edu}
\author{Shijin Zhang}
\address{School of Mathematics and Systems Science, Beihang University, Beijing 100191, P.R. China}
\email{shijinzhang@buaa.edu.cn}
\date{}							

\begin{document}
\maketitle
\begin{abstract}
\noindent
We study the behavior of the K\"ahler-Ricci flow on some Fano bundle which is a trivial bundle on one Zariski open set. We show that if the fiber is $\mathbb{P}^{m}$ blown up at one point or some weighted projective space blown up at the orbifold point and the initial metric is in a suitable k\"ahler class, then the fibers collapse in finite time and the metrics converge sub-sequentially in Gromov-Hausdorff sense to a metric on the base.
\end{abstract}
\section{Introduction}
The Ricci flow, introduced by Hamilton (\cite{H}), has become a powerful tool to study the topology and geometric structures of Riemannian manifolds. In general, the Ricci flow develops finite time singularities.  Hamilton's program of Ricci flow with surgeries was carried out by Perelman \cite{P1, P2, P3} to prove Thurston's geometrization conjecture. The minimal model theory in birational geometry can be viewed as the complex analogue of Thurston's geometrization conjecture. Later in \cite{Cao} Cao introduced the K\"ahler-Ricci flow and use it to prove the existence of K\"ahler-Einstein metrics on manifolds with negative or vanishing first Chern class (\cite{Yau, A}).

There are intensive study of the K\"ahler-Ricci flow in the past few years. In the general type case, has been studied by Tsuji \cite{Tsuji}, Tian-Zhang \cite{Tian-Zhang}, Guo \cite{Guo}, Guo-Song-Weinkove \cite{Guo-Song-Weinkove}, Tian-Zhang \cite{Tian-Zhangzl} and references therein. In the Calabi-Yau fiber case, has been studied by Song-Tian \cite{ST1}, Song-Weinkove \cite{SW3}, Fong-Zhang \cite{Fong-Zhang}, Gill \cite{Gill}, Tosatti-Weinkove-Yang \cite{TWY} and references therein. In the Fano case, has been studied by Chen-Tian \cite{Chen-Tian1,Chen-Tian2}, Chen-Wang \cite{Chen-Wang}, Phong-Sturm \cite{Phong-Sturm}, Phong-Song-Sturm-Weinkove \cite{PSSW1, PSSW2, PSSW3}, Sesum-Tian \cite{Sesum-Tian}, Sz\'ekelyhidi \cite{Szekely}, Tian-Zhang \cite{Tian-Zhangzl1}, Tian-Zhang-Zhang-Zhang \cite{TZZZ}, Tian-Zhu \cite{Tian-Zhu1,Tian-Zhu2} and references therein.

Moreover, Song-Tian have a nice observation which relate the K\"ahler-Ricci flow with birational geometry. In \cite{ST1, ST2, ST3} they introduced the analytic minimal model program which is parallel to Mori's birational minimal model program. On one hand, K\"ahler-Ricci flow with surgery can be viewed as the complex analogue of Thurston's three dimensional geometrization conjecture. On the other hand, the surgery is canonical and correspond to the birational surgery in Mori's program such as divisorial contraction or flip, see \cite{BCHM, HM}. In this article, our Fano fiber contraction in metric sense can be viewed as the analogue of Mori fiber space  in birational geometry.

In this article we study the behavior of the finite time singularity of the K\"ahler-Ricci flow. The behavior of the K\"ahler-Ricci flow with finite time singularity has been studied by  Feldman-Ilmanen-Knopf \cite{FIK}, Ilmanen-Knopf\cite{Ilmanen-Knopf}, Song-Weinkove \cite{SW0}, Song \cite{Song1,Song2},  Song-Sz\'ekelyhidi-Weinkove \cite{SSW}, Tian \cite{Tian1,Tian2}, Zhang \cite{Zhang1, Zhang2}, Fong \cite{Fong1, Fong2}, Collins-Tosatti \cite{Collins-Tosatti}, La Nave-Tian \cite{Nave-Tian}, Song-Yuan \cite{Song-Yuan}, Tosatti-Zhang \cite{Tosatti-Zhang2} and references therein.

Let $(M, g_{0})$ be a compact K\"ahler manifold of complex dimension $n\geq 2$. We write $\omega_{0}=\sqrt{-1}(g_{0})_{i\overline{j}}dz^{i}\wedge d\overline{z^{j}}$ for the K\"ahler form associated to $g_{0}$. We consider the K\"ahler-Ricci flow $\omega=\omega(t)$ given by
\begin{equation}\label{KRF}
\frac{\partial}{\partial t}\omega=-{\rm Ric}(\omega),\qquad \omega(0)=\omega_{0},
\end{equation}
where ${\rm Ric}(\omega)=-\sqrt{-1}\partial\overline{\partial}\log {\rm det}g$, where $g=g(t)$ is the metric associated to $\omega$.

It's well-known that, from Tian-Zhang \cite{Tian-Zhang}, a maximal smooth solution to (\ref{KRF}) exists on $[0, T)$ where $T>0$ is given by
\begin{equation}\label{MaximumTime}
T=\sup\{t>0|[\omega_{0}]-2\pi t c_{1}(X)>0\}.
\end{equation}

Song-Sz\'ekelyhidi-Weinkove \cite{SSW} studied the behavior of the K\"ahler-Ricci flow on the projective bundles. One essential point of their proof is that the projective space admits a metric which has positive bisectional curvature. In this article, we want to generalize their result in the sense that we could have more types of Fano fibers other than projective spaces. For example the fiber could be $\mathbb{P}^{m}$ blown-up at one point or $M_{m,k}$ which is the weighted projective space $Y_{m,k}(1\leq k<m)$ blown-up at the orbifold point (the definitions of $M_{m,k}$ and $Y_{m,k}$ see section 4).  In this paper, we consider the Fano bundle which is trivial on one Zariski open set. More precisely,

\begin{definition}[Fano bundle]\label{FanoBundle}
Let $X, Y$ be compact K\"ahler manifolds with dimension $n,m$, respectively, $F$ be a Fano manifold with dimension $n-m$ and a surjective holomorphic map $\pi: X\rightarrow Y$. We say $X$ is a Fano bundle over $Y$ with fiber $F$,  if for any $y\in Y$,  there exists a Zariski open set $y\in U\subset Y$ and a biholomorphism $\Phi: \pi^{-1}(U)\rightarrow U\times F$ such that the diagram
  $$\xymatrix{
    \pi^{-1}(U) \ar[rr]^{\Phi}\ar[dr]_{\pi} & & U\times F \ar[dl]^{{\rm Pr}_{1}} \\
     & U &
    }$$
commutes, where ${\rm Pr}_{1}$ is the projection map onto the first factor. We denote it as $(X,Y,\pi, F)$.
\end{definition}
\begin{remark}
We recall the definition of Fano fibration,  was defined by Tosatti-Zhang \cite{Tosatti-Zhang2}.  A compact K\"ahler manifold is said to admit a Fano fibration if there is a surjective holomorphic map $\pi:X\rightarrow Y$ with connected fibers, where $Y$ is a compact normal K\"ahler space with $0\leq {\rm dim}Y<{\rm dim}X$ and such that for every fiber $F$ of $\pi$ such that $-K_{X}|_{F}$ is ample. Hence the Fano bundle is a special Fano fibration.
\end{remark}
\begin{remark}
The simplest example of a Fano bundle is $X=F\times Y$ with a Fano manifold $F$ and any compact K\"ahler manifold $Y$. The projective bundle $X$ was considered by Song-Sz\'ekelyhidi-Weinkove \cite{SSW} is a Fano bundle with $F=\mathbb{P}^{n-m}$.
\end{remark}

Since the fiber $F$ is a Fano manifold, the solution $\omega(t)$ develops a singularity after a finite time. By (\ref{MaximumTime}), $T$ is finite since $F \cdot c_{1}(X)^{n-m}>0$ for every fiber $F$. Hence we assume that the limiting K\"ahler class $[\omega_{0}]-2\pi T c_{1}(X)$ satisfies
\begin{equation}\label{InitialData}
[\omega_{0}]-2\pi Tc_{1}(X)=[\pi^{*}\omega_{Y}]
\end{equation}
for some K\"ahler metric $\omega_{Y}$ on $Y$. It's well-known that the K\"ahler-Ricci flow equation (\ref{KRF}) is equivalent to the following complex Monge-Amp\`ere equation
\begin{equation}
\left\{
\begin{aligned}
\frac{\partial \varphi}{\partial t}&=\log\frac{(\frac{1}{T}((T-t)\omega_{0}+t \pi^{*}\omega_{Y})+\sqrt{-1}\partial\overline{\partial}\varphi)^{n}}{\Omega}\\
\varphi(0)&=0,\\
\end{aligned}
\right.
\end{equation}
where $\Omega$ is a smooth volume form, $\chi=\sqrt{-1}\partial\overline{\partial}\log\Omega \in -2\pi c_{1}(X)$, and $\omega(t)=\frac{1}{T}((T-t)\omega_{0}+t \pi^{*}\omega_{Y})+\sqrt{-1}\partial\overline{\partial}\varphi>0$. By Lemma \ref{LowerBound} below, we know there exists a bound function $\varphi_{T}$ on $X$ satisfying $\lim_{t\rightarrow T}\varphi(t)=\varphi_{T}$. We define
\begin{equation}\label{omegaT}
\omega_{T}:=\pi^{*}\omega_{Y}+\sqrt{-1}\partial\overline{\partial}\varphi_{T}\geq 0.
\end{equation}

Our first main result of this paper is that $\omega_{T}$ is bounded.
\begin{theorem}\label{MainThm0}
Assume $(X,Y,\pi, F)$ is a Fano bundle, $\omega_{0}$ is the K\"ahler metric on $X$, $\omega_{Y}$ is a K\"ahler metric on $Y$ satisfying (\ref{InitialData}) for some $T>0$, $\omega_{T}$ is defined by (\ref{omegaT}). Then there exists a uniform constant $C>0$ such that
\begin{equation}
C^{-1}\pi^{*}\omega_{Y}\leq \omega_{T}\leq C\pi^{*}\omega_{Y}.
\end{equation}
\end{theorem}

We also study the K\"ahler-Ricci flow on the Fano bundles with the fiber $F$ is $\mathbb{P}^{m}$ blown up at one point or $M_{m,k}$ which is the weighted projective space $Y_{m,k}$ (the definition see Section 4) blown up at the orbifold point.

The other main result of this paper shows that diameter of manifold $X$ with metric $\omega(t)$ is finite and there exists a sequence of metrics along the K\"ahler-Ricci flow converge subsequentially to a metric on $Y$ in the Gromov-Hausdorff sense as $t\rightarrow T$. It generalizes in some sense the result of Song-Sz\'ekelyhidi-Weinkove (Theorem 1.1 in \cite{SSW}).

\begin{theorem}\label{MainThm1}
Let $(X, Y, \pi, F)$ be a Fano bundle with $F$ is $\mathbb{P}^{m}$ blown up at one point $(m\geq 2)$ or $F=M_{m,k}(1\leq k<m)$,  $\omega_{Y}$ be a K\"ahler metric on $Y$ and $\omega_{0}$ be a K\"ahler metric on $X$. Assume $\omega(t)$ is a solution of the K\"ahler-Ricci flow (\ref{KRF}) for $t\in [0, T)$ with initial metric $\omega_{0}$ and $[\omega_{0}]-2\pi Tc_{1}(X)=[\pi^{*}\omega_{Y}]$, then we have
\begin{itemize}
\item[(1)] ${\rm diam}(X, \omega(t))\leq C$ for some uniform constant $C>0$;
\item[(2)] There exists a sequence of times $t_{i}\rightarrow T$ and a distance function $d_{Y}$ on $Y$ (which is uniformly equivalent to the distance induced by $\omega_{Y}$, such that $(X,\omega(t_{i}))$ converges to $(Y,d_{Y})$ in the Gromov-Hausdorff sense.
\end{itemize}
\end{theorem}
The Theorem \ref{MainThm1} is a combination of Theorem \ref{MainThm1a} and Theorem \ref{MainThm1b}.

When the dimension of $X$ is $2$, our method basically can cover most del Pezzo surface. It will be more interesting when the complex structure of the fiber is changing and when the fiber is general Fano variety in higher dimension.

In this paper, the notation 'tr' means that, if $\alpha=\sqrt{-1}\alpha_{i\overline{j}}dz^{i}\wedge d\overline{z^{j}}$ is a real $(1,1)$-form then we write
\begin{equation}
{\rm tr}_{\omega}\alpha=g^{i\overline{j}}\alpha_{i\overline{j}}=\frac{n\alpha\wedge \omega^{n-1}}{\omega^{n}}.
\end{equation}
In this notation, we can write $\Delta f={\rm tr}_{\omega}(\sqrt{-1}\partial\overline{\partial}f).$

In section 2, we recall some well known estimates for the K\"ahler-Ricci  flow, and using Song's argument in \cite{Song2} to establish a estimate of the horizontal level set, as an application, we prove Theorem \ref{MainThm0}. In Section 3, we using the estimate of the horizontal level set established in section 2 and the argument in \cite{SSW} to prove the case of fiber is $\mathbb{P}^{m}$ blown up at one point in Theorem \ref{MainThm1}. In Section 4, we recall the definitions of $M_{m,k}$ and $Y_{m,k}$, using He-Sun's theorem \cite{HS}  that any weighted projective space admits a orbifold K\"ahler metric with positive bisectional curvature,  Song-Weinkove's argument in \cite{SW2} and the argument of the proof in the case of $F$ is $\mathbb{P}^{m}$ blown up at one point, we can prove the case of $F=M_{m,k}$ in the Theorem \ref{MainThm1}.

\section{\textbf{The Main Estimates}}
In this section, we recall some estimates for the K\"ahler-Ricci flow, establish a estimate for $\omega(t)$ on the horizontal level set and prove the Theorem \ref{MainThm0}.

We define reference $(1,1)-$forms $\hat{\omega}_{t}$ on $X$ for $t\in[0, T]$ by
\begin{equation}
\hat{\omega}_{t}=\frac{1}{T}((T-t)\omega_{0}+t \pi^{*}\omega_{Y}).
\end{equation}
Then $\hat{\omega}_{t}$ is a K\"ahler form in the cohomology class $[\omega(t)]$ for $t\in [0, T)$. Let $\Omega$ be the unique smooth volume form on $X$ with $\sqrt{-1}\partial\overline{\partial}\log \Omega=\frac{\partial}{\partial t}\hat{\omega}_{t}=:\chi \in -2\pi c_{1}(X)$ and $\int_{X}\Omega=1.$ We also can write $\hat{\omega}_{t}$ as $\hat{\omega}_{t}=\omega_{0}+t\chi.$

It's well-known that the K\"ahler-Ricci flow equation (\ref{KRF}) is equivalent to the following complex Monge-Amp\`ere equation
\begin{equation}
\left\{
\begin{aligned}
\frac{\partial \varphi}{\partial t}&=\log\frac{(\hat{\omega}_{t}+\sqrt{-1}\partial\overline{\partial}\varphi)^{n}}{\Omega}\\
\varphi(0)&=0\\
\omega(t)&>0.\\
\end{aligned}
\right.
\end{equation}
where $\omega(t)=\hat{\omega}_{t}+\sqrt{-1}\partial\overline{\partial}\varphi$.

In this paper we use $C$ to denote a uniform constant, independent of time but possibly depending on $\omega_{0}, n, T$, which may differ from line to line. Then the following estimates are well known, see the Lemma 2.1 and Lemma 2.2 in \cite{SW1}.
\begin{lemma}\label{LowerBound}
For any K\"ahler manifold $(X,\omega_{0})$ and K\"ahler manifold $(Y,\omega_{Y})$. If there exists a surjective holomorphic map $\pi: X\rightarrow Y$, and the smooth solution $\omega(t)$ of the K\"ahler-Ricci flow (\ref{KRF}) on $X$ satisfying $\lim_{t\rightarrow T}[\omega(t)]=[\pi^{*}\omega_{Y} ](T<+\infty).$ Then we have
\quad
\begin{enumerate}
\item
There exists a uniform constant $C>0$ such that $||\varphi||_{L^{\infty}}\leq C$;
\item
There exists a uniform constant $C>0$ such that $\dot{\varphi}\leq C$;
\item
As $t\rightarrow T$, $\varphi(t)$ converges pointwise on $X$ to a bounded function $\varphi_{T}$ satisfying
\begin{equation}
\omega_{T}:=\pi^{*}\omega_{Y}+\sqrt{-1}\partial\overline{\partial}\varphi_{T}\geq 0.
\end{equation}
\item
There exists a uniform constant $c>0$ such that
\begin{equation}
\omega(t)\geq c\pi^{*}\omega_{Y}.
\end{equation}
\end{enumerate}
\end{lemma}

Next motivated by the argument of Song,  see subsection 3.1 in \cite{Song2}, which Song estimated the evolving metrics of the K\"ahler-Ricci flow in a well-chosen set of directions in the tangent space of each point on $X$ instead of all directions,   we estimate the metric $\omega(t)$ on the horizontal level set of the Fano bundle $X$.

Let $(X, Y, \pi, F)$ be the Fano bundle (see Definition \ref{FanoBundle}). Since for any $x\in X$, let $y=\pi(x)$, there exists a Zariski open set $(y\in) U\subset Y$, such that
the diagram
$$\xymatrix{
    \pi^{-1}(U) \ar[rr]^{\Phi}\ar[dr]_{\pi} & & U\times F \ar[dl]^{{\rm Pr}_{1}} \\
     & U &
    }$$
commutes, where ${\rm Pr}_{1}$ is the projection map onto the first factor.  Let $f={\rm Pr}_{2}\circ \Phi (x)$, $H=\Phi^{-1}(U\times\{f\})$, where ${\rm Pr}_{2}$ is the projection map onto the second factor.  Let  $D=Y\backslash U$ and $s$ be a holomorphic section on $[D]$ and let $h$ be a Hermitian metric on $[D]$. Define $|s|^{2}=hs\overline{s}$. Then on the horizontal level set $H$, we have the estimate for $\omega(t)$

\begin{lemma}\label{Hestimate}
Assume $\omega(t)$ is the solution of the K\"ahler-Ricci flow (\ref{KRF}) and $\lim_{t\rightarrow T}[\omega(t)]=[\pi^{*}\omega_{Y}]$. Fix any point $x\in X$,  then there exists $U\subset Y$,  let $f(x)={\rm Pr}_{2}\circ \Phi(x)$ and $H=\Phi^{-1}(U\times \{f(x)\})$. Then there exist uniform constants $C>0$ and $\alpha>0$, such that
$\omega(t)|_{H}\leq \frac{C}{\pi^{*}(|s|^{2\alpha})}(\pi^{*}\omega_{Y})|_{H}.$
\end{lemma}
\begin{proof}
Since for any $x\in\pi^{-1}(U)$, $\pi(x)=y\in U$,  and $\Phi$ is a biholomorphism from $\pi^{-1}(U)$ to $U\times F$, there exist constants $\alpha>0$ and $C>0$ such that $\pi^{*}\omega_{Y}|_{H}\geq \pi^{*}|s|^{\alpha}\omega_{Y}|_{U}$. On the other hand, for each time $t\in (0, T)$, $\omega(t)$ is equivalent to metric $\omega_{0}$. Hence if we let
\begin{center}
$u(t,x)=\pi^{*}(|s|^{2\alpha}){\rm tr}_{\pi^{*}\omega_{Y}|_{H}}(\omega(t)|_{H})(x),$
\end{center}
we know $u\rightarrow 0$ along $X\backslash \pi^{-1}(U)$ and hence a positive maximum must occur in $\pi^{-1}(U)$ at each fixed time $t\in (0, T)$. We assume the maximum can be obtained at point $x_{0}\in X$. Let $y_{0}=\pi(x_{0})\in Y$. We choose normal coordinate system $(z^{i})_{i=1,\cdots, n}$ for $g(t)$ at $x_{0}$ and $(w^{\alpha})_{\alpha=1,\cdots,m}$ for $g_{Y}$ at $y_{0}$. For any holomorphic vector $\frac{\partial}{\partial w^{\alpha}}$, there exist holomorphic vector $\frac{\partial}{\partial x^{\alpha}}\in T_{x}X$ such that $d\pi_{x}(\frac{\partial}{\partial x^{\alpha}})=\frac{\partial}{\partial w^{\alpha}}$ for any $x$ in the local normal coordinate chart of $x_{0}$. The map $\pi$ is given locally as $(\pi^{1},\cdots,\pi^{m})$ for holomorphic functions $\pi^{\alpha}=\pi^{\alpha}(z^{1},\cdots,z^{n})$. We write $\frac{\partial}{\partial x^{\alpha}}$ as $\frac{\partial}{\partial x^{\alpha}}=a_{\alpha}^{i}\frac{\partial}{\partial z^{i}}$ for holomorphic functions $a_{\alpha}^{i}$.  Then $u$ can be expressed as $u(t,x)=|s|^{2\alpha}(\pi(x))g_{Y}^{\alpha\overline{\beta}}a_{\alpha}^{i}\overline{a_{\beta}^{j}}g_{i\overline{j}}.$  For convenience, we denote $u_{1}=g_{Y}^{\alpha\overline{\beta}}a_{\alpha}^{i}\overline{a_{\beta}^{j}}g_{i\overline{j}}$.  Then at point $x_{0}$
\begin{eqnarray*}
\begin{aligned}
\Delta u_{1}&=g^{k\overline{l}}\partial_{k}\partial_{\overline{l}}(g_{Y}^{\alpha\overline{\beta}}a_{\alpha}^{i}\overline{a_{\beta}^{j}}g_{i\overline{j}})\\
&=\sum_{k,l=1}^{n}g^{k\overline{l}}\partial_{k}(\partial_{\overline{\delta}}g_{Y}^{\alpha\overline{\beta}}\overline{\pi^{\delta}_{l}}a_{\alpha}^{i}\overline{a_{\beta}^{j}}g_{i\overline{j}}
+g_{Y}^{\alpha\overline{\beta}}a_{\alpha}^{i}\overline{\partial_{l}a_{\beta}^{j}}g_{i\overline{j}}+g_{Y}^{\alpha\overline{\beta}}a_{\alpha}^{i}\overline{a_{\beta}^{j}}\partial_{\overline{l}}g_{i\overline{j}})\\
&=-\partial_{\gamma}\partial_{\overline{\delta}}(g_{Y})_{\beta\overline{\alpha}}\pi^{\gamma}_{k}\overline{\pi^{\delta}_{k}}a_{\alpha}^{i}\overline{a_{\beta}^{i}}
+|\partial_{k}a_{\alpha}^{i}|^{2}-a_{\alpha}^{i}\overline{a_{\alpha}^{j}}R_{i\overline{j}}\\
&=({\rm Rm}(g_{Y}))_{\gamma\overline{\delta}\beta\overline{\alpha}}\pi^{\gamma}_{k}\overline{\pi^{\delta}_{k}}a_{\alpha}^{i}\overline{a_{\beta}^{i}}
+|\partial_{k}a_{\alpha}^{i}|^{2}-a_{\alpha}^{i}\overline{a_{\alpha}^{j}}R_{i\overline{j}}
\end{aligned}
\end{eqnarray*}
On the other hand,
\begin{equation*}
\frac{\partial u_{1}}{\partial t}=g_{Y}^{\alpha\overline{\beta}}a_{\alpha}^{i}\overline{a_{\beta}^{j}}\frac{\partial}{\partial t}g_{i\overline{j}}=-a_{\alpha}^{i}\overline{a_{\alpha}^{j}}R_{i\overline{j}}.
\end{equation*}
Hence
\begin{eqnarray*}
\begin{aligned}
(\frac{\partial}{\partial t}-\Delta)\log u_{1}&=\frac{1}{u_{1}}(-({\rm Rm}(g_{Y}))_{\gamma\overline{\delta}\beta\overline{\alpha}}\pi^{\gamma}_{k}\overline{\pi^{\delta}_{k}}a_{\alpha}^{i}\overline{a_{\beta}^{i}}
-|\partial_{k}a_{\alpha}^{i}|^{2})+\frac{|\nabla u_{1}|^{2}}{u_{1}^{2}}\\
&\leq c_{Y}{\rm tr}_{\omega}\pi^{*}\omega_{Y}+\frac{1}{u_{1}}(\frac{|\nabla u_{1}|^{2}}{u_{1}}-|\partial_{k}a_{\alpha}^{i}|^{2}),
\end{aligned}
\end{eqnarray*}
where $-c_{Y}$ is a lower bound for the bisectional curvature of $\omega_{Y}$ on $Y$. It is easy to get (see \cite{ST1})
\begin{center}
$\frac{|\nabla u_{1}|^{2}}{u_{1}}-|\partial_{k}a_{\alpha}^{i}|^{2}\leq 0.$
\end{center}
Hence we have
\begin{equation}
(\frac{\partial}{\partial t}-\Delta)\log u_{1}\leq c_{Y}{\rm tr}_{\omega}\pi^{*}\omega_{Y}.
\end{equation}
Since $\sqrt{-1}\partial\overline{\partial}(\pi^{*}|s|^{2})(x_{0})=\sqrt{-1}\partial\overline{\partial}|s|^{2}(y_{0})$,  is bounded by some multiple of $\pi^{*}\omega_{Y}$.
Combine Lemma 2.1, we have
\begin{equation*}
(\frac{\partial}{\partial t}-\Delta)\log u\leq C{\rm tr}_{\omega}\pi^{*}\omega_{Y}\leq C'.
\end{equation*}
Hence by the maximum principle, we have $u\leq C$.
\end{proof}

Now we prove the Theorem \ref{MainThm0}.
\begin{proof}[Proof of Theorem \ref{MainThm0}]
Lower bound follows from (4) in Lemma \ref{LowerBound}. For any point $y\in Y$,  each fiber $\pi^{-1}(y)=F$ is a closed K\"ahler manifold, and since $\pi^{*}\omega_{Y}|_{\pi^{-1}(y)}=0$, we have
$$ \sqrt{-1}\partial\overline{\partial}\varphi_{T}|_{\pi^{-1}(y)}=\omega_{T}|_{\varphi_{T}}\geq 0,$$
since $\varphi_{T}$ is bounded, $\varphi_{T}$ must be constant on the fiber $\pi^{-1}(y)$. Hence there exists a bounded function $\psi_{T}$ on $Y$ satisfying
$$\varphi_{T}=\pi^{*}\psi_{T}.$$
Hence
$$\omega_{T}=\pi^{*}(\omega_{Y}+\sqrt{-1}\partial\overline{\partial}\psi_{T}).$$
Now for any $x\in X$, we may assume that $|s|^{2}(\pi(x))=0$, there exists an open set $\pi(x)\in U\subset Y$, such that Lemma \ref{Hestimate} holds. Now we consider the open set $U_{1/2}:=\{ y\in U| |s|^{2}(y)>1/2\}$. Then by Lemma \ref{Hestimate}, there exists a constant $C>0$ such that
$$\sqrt{-1}\partial\overline{\partial}\psi_{T}|_{U_{1/2}}\leq C\omega_{Y}.$$
Since $Y$ is a compact manifold, there exist a finite open set $\{U^{i}_{1/2}(1\leq i\leq N)\}$ ($N$ is a positive integer number) satisfying
$$ \cup_{i=1}^{N}U^{i}_{1/2}=Y.$$
Hence we obtain that there exists a uniform constant $C>0$ such that
$$ \sqrt{-1}\partial\overline{\partial}\psi_{T}\leq C\omega_{Y}.$$
Hence we finish the proof of the theorem.
\end{proof}

\section{\textbf{$F$ Is $\mathbb{P}^{m}$ Blown Up At One Point}}
In this section, we consider the case of $F$ is $\mathbb{P}^{m}$ blown up at one point. One essential point of Song-Sz\'ekelyhidi-Weinkove's proof \cite{SSW} is that the projective space admits a metric which has positive holomorphic bisectional  curvature. Although $\mathbb{P}^{m}$ blown up at one point doesn't admit a metric with nonnegative holomorphic bisectional curvature, but we have such metric with nonnegative bisectional curvature on outside of the divisor. Then we need to estimate the locally holomorphic vector field near the divisor under the evolving metrics,  by using a idea of Song-Weinkove \cite{SW1}.  We also need Lemma 2.2,  estimate of the evolving metrics along the K\"ahler-Ricci flow which were restricted to a horizontal set. We prove the following
\begin{theorem}\label{MainThm1a}
Let $(X, Y, \pi, F)$ be a Fano bundle with $F$ is $\mathbb{P}^{m}$ blown up at one point $(m\geq 2)$,  $\omega_{Y}$ be a K\"ahler metric on $Y$ and $\omega_{0}$ be a K\"ahler metric on $X$. Assume $\omega(t)$ is a solution of the K\"ahler-Ricci flow (\ref{KRF}) for $t\in [0, T)$ with initial metric $\omega_{0}$ and $[\omega_{0}]-2\pi Tc_{1}(X)=[\pi^{*}\omega_{Y}]$, then we have
\begin{itemize}
\item[(1)] ${\rm diam}(X, \omega(t))\leq C$ for some uniform constant $C>0$;
\item[(2)] There exists a sequence of times $t_{i}\rightarrow T$ and a distance function $d_{Y}$ on $Y$ (which is uniformly equivalent to the distance induced by $\omega_{Y}$, such that $(X,\omega(t_{i}))$ converges to $(Y,d_{Y})$ in the Gromov-Hausdorff sense.
\end{itemize}
\end{theorem}

\subsection{Key Estimates}
Write $\pi_{1}: F\rightarrow \mathbb{P}^{m}$ for the blow-down map, which is an isomorphism from $F \backslash E$ to $\mathbb{P}^{m}\backslash \{p\}$, where $p\in \mathbb{P}^{m}$ and $E=\pi_{1}^{-1}(p)$, which is biholomorphic to $\mathbb{P}^{m-1}$. For convenient, once and for all, a coordinate chart $V$ centered at $p$, which we identify via coordinates $z^{1}, \cdots, z^{m}$ with the unit ball $D_{1}$ in $\mathbb{C}^{m}$,
\begin{equation}
D_{1}=\{(z^{1},\cdots, z^{m})\in \mathbb{C}^{m}|\sum_{i=1}^{m}|z^{i}|^{2}<1\}.
\end{equation}
Denote by $g_{e}$ the Euclidean metric on $D_{1}$. Since $g_{e}$ and $g_{FS}$ are uniformly equivalent on $D_{1}$, it suffices to estimates for $g_{e}$ on $D_{1}$. Write $D_{r}\subset D_{1}$ for the ball of radius $0<r<1$ with respect to $g_{e}.$

We recall the definition of the blow-up construction, following the exposition in \cite{GH}. We identify $\pi_{1}^{-1}(D_{1})$ with the submanifold $\tilde{D_{1}}$ of $D_{1}\times \mathbb{P}^{m-1}$ given by
\begin{equation}
\tilde{D_{1}}=\{(z,l) \in D_{1}\times \mathbb{P}^{m-1} | z^{i}l^{j}=z^{j}l^{i}\},
\end{equation}
where $l=[l^{1},\cdots,l^{m}]$ are homogeneous coordinates on $\mathbb{P}^{m-1}$. The map $\pi_{1}$ restricted to $\tilde{D_{1}}$ is the projection $\pi|_{\tilde{D_{1}}}(z,l)=z\in D_{1}$, with the exceptional divisor $E\cong\mathbb{P}^{m-1}$ given by $\pi_{1}^{-1}(0)$. The map $\pi$ gives an isomorphism from $\tilde{D_{1}}\backslash E$ onto the punctured ball $D_{1}\backslash \{0\}$.

On $\tilde{D_{1}}$ we have coordinate charts $\tilde{D_{1}}_{i}=\{l^{i}\neq 0\}$ with local coordinates $\tilde{z}(i)^{1},\cdots,\tilde{z}(i)^{m}$ given by $\tilde{z}(i)^{j}=l^{j}/l^{i}=z^{j}/z^{i}$ for $j\neq i$ and $\tilde{z}(i)^{i}=z^{i}$. The divisor $E$ is given in $\tilde{D_{1}}_{i}$ by $\{\tilde{z}(i)^{i}=0\}$. The line bundle $[E]$ over $\tilde{D_{1}}$ has transition functions $z^{i}/z^{j}$ on $\tilde{D_{1}}_{i}\cap\tilde{D_{1}}_{j}$. We can define a global section $s$ of $[E]$ over $F$ by setting $s(z)=z^{i}$ on $\tilde{D_{1}}_{i}$ and $s=1$ on $F\backslash \pi_{1}^{-1}(D_{1/2})$. The section $s_{1}$ vanishes along the exceptional divisor $E$. We also define a Hermitian metric $h_{1}$ on $[E]$ as follows. First let $h_{2}$ be the Hermitian metric on $[E]$ over $\tilde{D_{1}}$ given in $\tilde{D_{1}}_{i}$ by
\begin{equation}
h_{2}=\frac{\sum_{j=1}^{m}|l^{j}|^{2}}{|l^{i}|^{2}},
\end{equation}
and let $h_{3}$ be the Hermitian metric on $[E]$ over $F\backslash E$ determined by $|s_{1}|_{h_{2}}^{2}=1$. Now define the Hermitian metric $h_{1}$ by $h_{1}=\rho_{1}h_{2}+\rho_{2}h_{3}$, where $\rho_{1}, \rho_{2}$ is a partition of unity for the cover $(\pi_{1}^{-1}(D_{1}), F\backslash \pi_{1}^{-1}(D_{1/2}))$ of $F$, so that $h_{1}=h_{2}$ on $\pi_{1}^{-1}(D_{1/2})$. The function $|s_{1}|_{h_{1}}^{2}$ on $F$ is given on $\pi_{1}^{-1}(D_{1/2})$ by
\begin{equation}\label{r}
|s_{1}|_{h_{1}}^{2}(x)=\sum_{i=1}^{m}|z^{i}|^{2}=: r^{2},
\end{equation}
for $\pi_{1}(x)=(z^{1},\cdots, z^{m})$. On $\pi_{1}^{-1}(D_{1/2}\backslash \{0\})$, the curvature $R(h_{1})$ of $h_{1}$ is given by
\begin{equation}
R(h_{1})=-\sqrt{-1}\partial\overline{\partial}\log(\sum_{i=1}^{m}|z^{i}|^{2}).
\end{equation}

We have the following lemma (see \cite{GH}, p.187).
\begin{lemma}\label{omegaF}
For sufficiently small $\epsilon_{0}>0$,
\begin{equation}
\omega_{F}=\pi_{1}^{*}\omega_{FS}-\epsilon_{0}R(h_{1})
\end{equation}
is a K\"ahler form on $F$.
\end{lemma}

From now on we fix $\epsilon_{0}>0$ as in the Lemma \ref{omegaF}, with $\omega_{F}$ defined in Lemma \ref{omegaF}. In $\pi_{1}^{-1}(D_{1/2}\backslash \{0\})$, which we can identify with $D_{1/2}\backslash \{0\}$, the metric $\omega_{F}$ has the form:
\begin{equation}
\omega_{F}=\pi_{1}^{*}\omega_{FS}+\sqrt{-1}\frac{\epsilon_{0}}{r^{2}}\sum_{i,j=1}^{m}(\delta_{ij}-\frac{\overline{z^{i}}z^{j}}{r^{2}})dz^{i}d\overline{z^{j}},
\end{equation}
for  $r$ given by (\ref{r}). It is easy to see that, in $D_{1/2}\backslash \{0\}$, $R(h_{1})\leq 0$, and the following lemma holds (see \cite{SW1}).

\begin{lemma}\label{2.4}
There exist positive constants $C$ such
\begin{equation}
\pi_{1}^{*}\omega_{FS}\leq \omega_{F}\leq C\frac{\pi_{1}^{*}\omega_{FS}}{|s_{1}|_{h_{1}}^{2}}
\end{equation}
\end{lemma}

Since $(X,Y, \pi, F)$ is a Fano bundle, for any $x\in X$, let $y=\pi(x)$, there exists a Zariski open set $(y\in) U\subset Y$, such that
the diagram
$$\xymatrix{
    \pi^{-1}(U) \ar[rr]^{\Phi}\ar[dr]_{\pi} & & U\times F \ar[dl]^{{\rm Pr}_{1}} \\
     & U &
    }$$
commutes, where ${\rm Pr}_{1}$ is the projection map onto the first factor. Let  $D=Y\backslash U$ and $s$ be a holomorphic section on $[D]$ and let $h$ be a Hermitian metric on $[D]$. Define $|s|^{2}=hs\overline{s}$, for simplicity, we also write $\pi^{*}|s|_{h}^{2}$ as $|s|_{h}^{2}$. On $\pi^{-1}(U)$, we denote $\tilde{\omega}=\Phi^{*}({\rm Pr}_{2}^{*}\pi_{1}^{*}\omega_{FS}+{\rm Pr}_{1}^{*}\omega_{Y})$, we also write $|s_{1}|_{h_{1}}^{2}$ to represent $\Phi^{*}{\rm Pr}_{2}^{*}(|s_{1}|_{h_{1}}^{2})$, where ${\rm Pr}_{2}$ is the projection map onto the second factor.  Then we have the following
\begin{lemma}\label{2.5}
There exist uniform constants $C>0$ and $\alpha>0$ such that for $\omega=\omega(t)$ a solution of the K\"ahler-Ricci flow,
\begin{equation}
\omega(t)\leq \frac{C}{|s|_{h}^{2\alpha}|s_{1}|_{h_{1}}^{2}}\tilde{\omega}.
\end{equation}
\end{lemma}
\begin{proof}
Fix $0<\epsilon\leq 1$.  By Lemma \ref{2.4}, we know
$$\tilde{\omega}\geq C \Phi^{*}(|s_{1}|_{h_{1}}^{2}({\rm Pr}_{2}^{*}\omega_{F}+{\rm Pr}_{1}^{*}\omega_{Y})).$$

Since ${\rm Pr}_{2}^{*}\omega_{F}+{\rm Pr}_{1}^{*}\omega_{Y}$ is a fixed K\"ahler metric on $U\times F$ and $\Phi$ is a biholomorphism from $\pi^{-1}(U)$ to $U\times F$, for any fixed time $t$, there exists a constant $\alpha>0$ such that
$${\rm tr}_{\tilde{\omega}}\omega \leq \frac{C}{|s|_{h}^{\alpha}|s_{1}|_{h_{1}}^{2}}.$$

Hence if we set
\begin{equation}
Q_{\epsilon}=\log(|s|_{h}^{2\alpha}|s_{1}|_{h_{1}}^{2+2\epsilon}{\rm tr}_{\tilde{\omega}}\omega).
\end{equation}
For each fixed time $t\in (0, T)$. We know the maximum of $Q_{\epsilon}$ must be obtained at some point $x_{0}\in\Phi^{-1}(U\times F\backslash E)$.  Now we compute at point $(x_{0}, t)$
\begin{eqnarray}
\begin{aligned}
(\frac{\partial}{\partial t}-\Delta)Q_{\epsilon}&=(\frac{\partial}{\partial t}-\Delta)\log {\rm tr}_{\tilde{\omega}}\omega+\alpha{\rm tr}_{\omega}R(h)+(1+\epsilon){\rm tr}_{\omega}R(h_{1})\\
&\leq (\frac{\partial}{\partial t}-\Delta)\log {\rm tr}_{\tilde{\omega}}\omega+\alpha{\rm tr}_{\omega}R(h).
\end{aligned}
\end{eqnarray}
From the argument in the proof of Lemma \ref{Hestimate},  there exists a uniform constant $C>0$ such that
$$\alpha{\rm tr}_{\omega}R(h)\leq C{\rm tr}_{\omega}\pi^{*}\omega_{Y}\leq C'.$$
By a well-known computation (see \cite{Yau, A, Cao}):
\begin{eqnarray}\label{2.25}
\begin{aligned}
(\frac{\partial}{\partial t}-\Delta)\log {\rm tr}_{\tilde{\omega}}\omega
&=\frac{1}{{\rm tr}_{\tilde{\omega}}\omega}(-g^{i\overline{j}}\tilde{g}^{k\overline{q}}\tilde{g}^{p\overline{l}}g_{k\overline{l}}\tilde{R}_{i\overline{j}p\overline{q}}-g^{i\overline{j}}\tilde{g}^{k\overline{l}}g^{p\overline{q}}\tilde{\nabla}_{i}g_{k\overline{q}}\tilde{\nabla}_{\overline{j}}g_{p\overline{l}}+\frac{|\nabla {\rm tr}_{\tilde{\omega}}\omega|^{2}}{{\rm tr}_{\tilde{\omega}}\omega})\\
&\leq -\frac{1}{{\rm tr}_{\tilde{\omega}}\omega} g^{i\overline{j}}\tilde{g}^{k\overline{q}}\tilde{g}^{p\overline{l}}g_{k\overline{l}}\tilde{R}_{i\overline{j}p\overline{q}}
\end{aligned}
\end{eqnarray}
Denote $\hat{g}$ as the product metric ${\rm Pr}_{2}^{*}\pi_{1}^{*}\omega_{FS}+{\rm Pr}_{1}^{*}\omega_{Y}$, then $\tilde{g}=\Phi^{*}\hat{g}$. We compute with metric $\hat{g}$, since the bisectional curvature of $\omega_{FS}$ is positive, we have
\begin{equation}
\hat{R}_{i\overline{j}p\overline{q}}=({\rm Pr}_{2}^{*}\pi_{1}^{*}R(\omega_{FS}))_{i\overline{j}p\overline{q}}+({\rm Pr}_{1}^{*}R(\omega_{Y}))_{i\overline{j}p\overline{q}}\geq ({\rm Pr}_{1}^{*}R(\omega_{Y}))_{i\overline{j}p\overline{q}}.
\end{equation}
Since $\Phi^{*}{\rm Pr}_{1}^{*}=\pi^{*}$, pulling back via the map $\Phi$, we have
$$g^{i\overline{j}}\tilde{g}^{k\overline{q}}\tilde{g}^{p\overline{l}}g_{k\overline{l}}\tilde{R}_{i\overline{j}p\overline{q}}\geq -C({\rm tr}_{\tilde{\omega}}\omega)({\rm tr}_{\omega}\pi^{*}\omega_{Y})$$
for some uniform constant. Hence we obtain that
\begin{equation*}
(\frac{\partial}{\partial t}-\Delta)\log {\rm tr}_{\tilde{\omega}}\omega\leq C{\rm tr}_{\omega}{\rm Pr}_{1}^{*}\omega_{Y}\leq C'.
\end{equation*}
Hence
\begin{equation}
(\frac{\partial}{\partial t}-\Delta)Q_{\epsilon}\leq C.
\end{equation}
Then using the maximum principle and letting $\epsilon\rightarrow 0$, we obtain the lemma.
\end{proof}

We assume that $|s|_{h}(y)=1$ and denote $U_{1/2}=\{\tilde{y}\in U||s|_{h}^{2}(y)>1/2\}$.

Consider the holomorphic vector field
$$\sum_{i}^{m}z^{i}\frac{\partial}{\partial z^{i}},$$
defined on the unit ball $D_{1}$. This defines via $\pi_{1}$ a holomorphic vector field $V$ on $\pi_{1}^{-1}(D_{1})\subset F$ which vanishes to order $1$ along the exceptional divisor $E$. We can extend $V$ to be a smooth $T^{1,0}$ vector field on the whole of $F$, and ${\rm Pr}_{2}^{*}(V)$ to be a smooth $T^{1,0}$ vector field on $U_{1/2}\times F$,  then pull back by $\Phi$ and then extend it to a vector $\tilde{V}$ on whole of $X$ which vanish on $X\backslash \pi^{-1}(U_{1/2})$. We then have the following lemma.
\begin{lemma}\label{2.6}
For $\omega=\omega(t)$ a solution of the K\"ahler-Ricci flow, we have the estimate
\begin{equation}
|\tilde{V}|_{\omega}^{2}\leq C|s_{1}|_{h_{1}},
\end{equation}
for a uniform constant $C$. Locally, in $D_{1/2}\backslash \{0\}$ we have
\begin{equation}
|W|_{g}^{2}\leq \frac{C}{r},
\end{equation}
for
$$ W=\sum_{i=1}^{m}(\frac{x^{i}}{r}\frac{\partial}{\partial x^{i}}+\frac{y^{i}}{r}\frac{\partial}{\partial y^{i}})$$
the unit length radial vector field with respect to $g_{e}$, where $z^{i}=x^{i}+\sqrt{-1}y^{i}.$
\end{lemma}
In the statement and proof of the lemma, we are identifying $\pi_{1}^{-1}(D_{1/2}\backslash \{0\})$ with $D_{1/2}\backslash \{0\}$ via the map $\pi_{1}$, writing $g$ for the K\"ahler metric $(\pi_{1}^{-1})^{*}(\omega |_{F})$.

\begin{proof}
In this proof, we denote $\omega_{U}$ for $\Phi^{*}({\rm Pr}_{1}^{*}\omega_{Y}+{\rm Pr}_{2}^{*}\omega_{F})$ on $\pi^{-1}(U)$ and a Hermitian metric $\tilde{\omega}=\rho_{1}\omega_{0}+\rho_{2}\omega_{U}$, where $\rho_{1}, \rho_{2}$ is a partition of unity for the cover $(X\backslash \pi^{-1}(U_{1/2}),\pi^{-1}(U))$, so that $\tilde{\omega}=\omega_{U}$ on $\pi^{-1}(U_{1/2})$ and  which is equivalent to metric $\omega_{0}$.

From the Lemma \ref{omegaF} we have, in $D_{1/2}$,
\begin{equation}
|\tilde{V}|_{\tilde{\omega}}^{2}=|V|_{\omega_{F}}^{2}=|V|_{\pi^{*}\omega_{FS}}^{2}.
\end{equation}
It follows that $|\tilde{V}|_{\omega_{0}}^{2}$ is uniformly equivalent to $|s_{1}|_{h_{1}}^{2}=r^{2}$ in $D_{1/2}$.

For any fixed point $(x,t)$. We compute in a normal coordinate system for $g$ at $(x,t)$, we have
\begin{equation}\label{2.41}
(\frac{\partial}{\partial t}-\Delta)\log|\tilde{V}|_{\omega}^{2}=\frac{1}{|\tilde{V}|_{\omega}^{2}}(-g^{i\overline{j}}g_{k\overline{l}}(\partial_{i}\tilde{V}^{k})(\overline{\partial_{j}\tilde{V}^{l}})+\frac{|\nabla|\tilde{V}|_{\omega}^{2}|_{\omega}^{2}}{|\tilde{V}|_{\omega}^{2}})\leq 0.
\end{equation}
Where we use the Cauchy-Schwarz inequality to get the above inequality (the detail, see the proof of Lemma 2.6 in \cite{SW1}).

Then using the maximum principle, we obtain that there exists a positive constant $C$ such that
\begin{equation}\label{2.37}
|\tilde{V}|_{\omega}^{2}\leq C|s_{1}|_{h_{1}}^{2}.
\end{equation}
Now define a Hermitian metric $\tilde{\omega}_{F}$ on $\mathbb{P}^{m}$ by
\begin{center}
$\tilde{\omega}_{F}=\omega_{e}$ \quad on \quad $D_{1}$,
\end{center}
and extending in an arbitrary way to be a smooth Hermitian metric on $F$. For small $\epsilon>0$, we consider the quantity
\begin{equation}
Q_{\epsilon}=\log(|\tilde{V}|_{\omega}^{2+2\epsilon}|s|_{h}^{2\alpha+2\epsilon}{\rm tr}_{\Phi^{*}({\rm Pr}_{2}^{*}\pi_{1}^{*}\tilde{\omega}_{F}+{\rm Pr}_{1}^{*}\omega_{Y})}\omega)-At
\end{equation}
where $\alpha$ is the constant in Lemma \ref{2.5} and $A$ is a constant to be determined. Since $\tilde{\omega}_{F}$ is uniformly equivalent to $\omega_{FS}$, we see that for fixed $t$, using Lemma \ref{2.5} and (\ref{2.37}),
$$(|V|_{\omega}^{2+2\epsilon}|s|_{h}^{2\alpha+2\epsilon}{\rm tr}_{\Phi^{*}({\rm Pr}_{2}^{*}\pi_{1}^{*}\tilde{\omega}_{F}+{\rm Pr}_{1}^{*}\omega_{Y})}\omega)(x)$$
tends to zero as $x$ tends to $X\backslash \pi^{-1}(U_{1/2})\cup \Phi^{-1}(U_{1/2}\times E)$ and thus $Q_{\epsilon}(x)$ tends to negative infinity. We now applying the maximum principle to $Q_{\epsilon}$. Since $Q_{\epsilon}$ is uniformly bounded from above on $\Phi^{-1}(U_{1/2}\times F)\backslash \Phi^{-1}(U_{1/2}\times D_{1/2}\backslash \{0\})$, we only need to rule out the case when $Q_{\epsilon}$ attains its maximum at a point in $ \Phi^{-1}(U_{1/2}\times D_{1/2}\backslash \{0\})$.  Assume that at some point $(x_{0},t_{0})\in  \Phi^{-1}(U_{1/2}\times D_{1/2}\backslash \{0\})\times (0, T)$, we have $\sup_{ \Phi^{-1}(U\times F\backslash E)\times [0, t_{0}]}Q_{\epsilon}=Q_{\epsilon}(x_{0},t_{0}).$

As in the proof of Lemma \ref{2.5}, we have
\begin{equation}
(\frac{\partial}{\partial t}-\Delta)\log |s|_{h}^{2\alpha+2\epsilon}=(\alpha+\epsilon){\rm tr}_{\omega}R(h)\leq C{\rm tr}_{\omega}\pi^{*}\omega_{Y}.
\end{equation}
By (\ref{2.25}), pulling back by the biholomorphic map $\Phi$, we have in $ \Phi^{-1}(U_{1/2}\times D_{1/2}\backslash \{0\}),$
\begin{equation}
(\frac{\partial}{\partial t}-\Delta)\log({\rm tr}_{\Phi^{*}({\rm Pr}_{2}^{*}\pi_{1}^{*}\tilde{\omega}_{F}+{\rm Pr}_{1}^{*}\omega_{Y})}\omega)\leq C{\rm tr}_{\omega}\Phi^{*}{\rm Pr}_{1}^{*}\omega_{Y}.
\end{equation}
Hence
\begin{equation}
(\frac{\partial}{\partial t}-\Delta)\log |s|_{h}^{2\alpha+2\epsilon}({\rm tr}_{\Phi^{*}({\rm Pr}_{2}^{*}\pi_{1}^{*}\tilde{\omega}_{F}+{\rm Pr}_{1}^{*}\omega_{Y})}\omega)\leq C{\rm tr}_{\omega}\pi^{*}\omega_{Y}\leq C'.
\end{equation}
Take $A=C'+1.$ By (\ref{2.41}), at $(x_{0}, t_{0})$, we obtain
\begin{equation}
0\leq (\frac{\partial}{\partial t}-\Delta)Q_{\epsilon}\leq -1,
\end{equation}
a contradiction. Thus $Q_{\epsilon}$ is uniformly bounded from above. Letting $\epsilon$ tend to zero, since $|s|_{h}^{2}>1/2$ on $\pi^{-1}(U_{1/2})$, we obtain
\begin{equation}
({\rm tr}_{\Phi^{*}({\rm Pr}_{2}^{*}\pi_{1}^{*}\omega_{FS}+{\rm Pr}_{1}^{*}\omega_{Y})}\omega)|\tilde{V}|_{\omega}^{2}\leq C,
\end{equation}
for some uniform constant $C$. By Lemma \ref{2.4}, we have
\begin{equation*}
({\rm tr}_{\omega_{0}}\omega)|\tilde{V}|_{\omega}^{2}\leq C,
\end{equation*}
and since $|\tilde{V}|_{\omega}^{2}\leq({\rm tr}_{\omega_{0}}\omega)|\tilde{V}|_{\omega_{0}}^{2}$ this gives
\begin{equation}
|\tilde{V}|_{\omega}^{4}\leq C|\tilde{V}|_{\omega_{0}}^{2},
\end{equation}
then the lemma follows from the fact that $|\tilde{V}|_{\omega_{0}}^{2}$ is uniformly equivalent to $|s_{1}|_{h_{1}}^{2}$ in $D_{1/2}$.
\end{proof}

Next we estimate on the lengths of spherical and radial paths in the punctured ball $D_{1/2}\backslash \{0\}$, which again we identify with its preimage in each fiber under $\pi_{1}$.
\begin{lemma}\label{2.7}
We have
\begin{itemize}
\item[(1)] For any $y\in Y$ and for $0<r<1/2$, the diameter of the $2m-1$ sphere $S_{r}$ of radius $r$ in $D_1$ centered at the origin with the metric induced from $\omega|_{\pi^{-1}(y)}$ is uniformly bounded from above, independent of $r$ and $y$.
\item[(2)] For any $z\in D_{1/2}\backslash \{0\}$, the length of a radial path $\gamma(\lambda)=\lambda z$ for $\lambda\in (0, 1]$ with respect to $\omega|_{\pi^{-1}(y)}$ is uniformly bounded from above by a uniform constant multiple of $|z|^{1/2}$.
\end{itemize}
Hence the diameter of $D_{1/2}\backslash \{0\}$ with respect to $\omega|_{\pi^{-1}(y)}$ is uniformly bounded from above and
$$ {\rm diam}(\pi^{-1}(y), \omega|_{\pi^{-1}(y)})\leq C.$$
\end{lemma}
\begin{proof}
For any $y\in Y$, we can choose $|s|_{h}^{2}(y)=1$. Then using Lemma \ref{2.5}, consider the metric $\omega|_{\pi^{-1}(y)}$, we have
\begin{equation}
\omega|_{\pi^{-1}(y)}\leq \frac{C}{|s_{1}|_{h_{1}}^{2}}\pi_{1}^{*}\omega_{FS}.
\end{equation}
Then using the same argument in the proof of Lemma 2.7 in \cite{SW1}, we obtain the lemma.
\end{proof}

Now we can prove the (1) of Theorem \ref{MainThm1a}.
\begin{proof}[Proof of (1) in Theorem \ref{MainThm1a}]
For any $p,q\in X$. Denote $p_{1}=\pi(p), q_{1}=\pi(q)$.  Then there exist two open subset $U_{1}, U_{2}\subset Y$, such that $p_{1}\in U_{1}, p_{2}\in U_{2}$  and there exist holomorphic maps $\Phi_{1}, \Phi_{2}$ such that $\Phi_{i}: \pi^{-1}(U_{i})\rightarrow U_{i}\times F (i=1,2)$ are the biholomorphic map. Denote $p_{2}={\rm Pr}_{2}\Phi_{1}(p), q_{2}={\rm Pr}_{2}\Phi_{2}(q)$. Since $Y$ is compact, we may assume $U_{1}\cap U_{2}\neq \varnothing$. We assume $\tilde{p}_{1}\in U_{1}\cap U_{2}$, denote $\tilde{p}=\Phi_{1}^{-1}((\tilde{p}_{1}, p_{2}))$ and $\tilde{q}=\Phi_{2}^{-1}((\tilde{p}_{1}, q_{2}))$, by Lemma \ref{Hestimate} we know $d_{\omega(t)}(p, \tilde{p})\leq C$ and $d_{\omega(t)}(q, \tilde{q})\leq C$ .

Since $\tilde{p}, \tilde{q}\in \pi^{-1}(\tilde{p}_{1})$, then by Lemma \ref{2.7}, we have $d_{\omega(t)}(\tilde{p},\tilde{q})\leq C$. Using the triangle inequality  we finish the proof of (1) in Theorem \ref{MainThm1}.
\end{proof}

\subsection{Diameter of fiber tends to zero}
In this subsection, we prove that the diameter of fiber with $\omega(t)$ tends to zero as $t\rightarrow T$. Let $d_{\omega}=d_{\omega(t)}$ be the distance function on $X$ associated to the evolving K\"ahler metric $\omega$. Using the same argument in the proof of Lemma 3.2 and Lemma 3.3 in \cite{SW1}, we prove the following lemmas.
\begin{lemma}\label{3.2}
Fix a point $y_{0}\in Y$. There exists a uniform constant $C$ (independent of $y_{0}$) such that for any $p,q\in E$, and any $t\in [0, T)$,
\begin{equation}
d_{\omega}(\Phi^{-1}(y_{0}, p),\Phi^{-1}(y_{0}, q))\leq C(T-t)^{1/3}.
\end{equation}
\end{lemma}
\begin{proof}
We can assume that $|s|_{h}^{2}(y_{0})=1$. We replace $E$ by $\Phi^{-1}(\{y_{0}\}\times E)$ in the proof of Lemma 3.2 in \cite{SW1}, using Lemma \ref{2.5} and using Lemma \ref{2.7}. See the argument of the proof of Lemma 3.2 in \cite{SW1}.
\end{proof}

Combine Lemma \ref{2.7} and Lemma \ref{3.2}, we have
\begin{lemma}\label{3.3}
Fix a point $y_{0}\in Y$. There exists a uniform constant $C$ (independent of $y_{0}$) such that for any $0<\delta_{0}<1/2$ and for any$t\in [0,T)$
\begin{equation}
{\rm diam}_{\omega(t)}(\Phi^{-1}(\{y_{0}\}\times \pi_{1}^{-1}(D_{\delta_{0}})))<C(|\delta_{0}|^{1/2}+(T-t)^{1/3}).
\end{equation}
\end{lemma}
\begin{proof}
We also assume that $|s|_{h}^{2}(y_{0})=1$. For any $p, q\in\pi_{1}^{-1}(D_{\delta_{0}}))$. Since Lemma \ref{3.2}, we only consider $p\in \pi_{1}^{-1}(D_{\delta_{0}}\backslash \{0\}))$ and $q\in E$. By Lemma \ref{2.7} (2), we know the length of a radial path $\gamma(\lambda)=\lambda p$  with respect to $\omega$ is uniformly bounded from above by $C|p|^{1/2}\leq C\delta_{0}^{1/2}$. Since $\gamma(\lambda)$ tends to a point $p_{0}\in E$ as $\lambda\rightarrow 0^{+}$, we know $d_{\omega|_{\pi^{-1}(y_{0})}}(p,p_{0})\leq C\delta_{0}^{1/2}$. By Lemma \ref{3.2}, $d_{\omega}(\Phi^{-1}((y_{0},p_{0})), \Phi^{-1}((y_{0},q)))\leq C(T-t)^{1/3}$. Hence we have
\begin{equation*}
d_{\omega}(\Phi^{-1}((y_{0},p)), \Phi^{-1}((y_{0},q)))\leq C(\delta_{0}^{1/2}+(T-t)^{1/3}).
\end{equation*}
\end{proof}

\begin{lemma}\label{EstimateFiber}
Fix a point $y_{0}\in Y$. There exists a uniform constant $C$ (independent of $y_{0}$) such that for any $p,q\in \pi^{-1}(y_{0})$, and any $t\in [0, T)$,
\begin{equation}
d_{\omega}(p,q)\leq C(T-t)^{1/15}.
\end{equation}
\end{lemma}
\begin{proof}
We also assume that $|s|_{h}^{2}(y_{0})=1$. For each fixed $t$ satisfying $T-t<2^{-15}$, i.e., $2(T-t)^{2/15}<1/2$. Take $\delta_{0}=(T-t)^{2/15}$, by Lemma \ref{3.3}, we have
 $${\rm diam}_{\omega(t)}(\Phi^{-1}(\{y_{0}\}\times \pi_{1}^{-1}(D_{2\delta_{0}})))<C((T-t)^{1/15}+(T-t)^{1/3})\leq C'(T-t)^{1/15}.$$

 We denote $p'={\rm Pr}_{2}\circ \Phi(p), q'={\rm Pr}_{2}\circ \Phi(p)$. Hence we only consider the case of $p', q'\in F\backslash \pi_{1}^{-1}(D_{\delta_{0}})$.

 Since $\pi^{-1}(y_{0})$ is biholomorphic to $F$, which is $\mathbb{P}^{m}$ blown up at one point, there exists a curve $\gamma\cong \mathbb{P}^{1}$, such that $p, q\in \gamma\cap (\pi^{-1}(y_{0})\backslash \Phi^{-1}(\{y_{0}\}\times\pi_{1}^{-1}(D_{\delta_{0}})))$. We may assume that $p, q$ lie in a fixed coordinate chart $U$ whose image under the holomorphic coordinate $z=x+\sqrt{-1}y$ is a ball of radius $2$ in $\mathbb{C}=\mathbb{R}^{2}$ with respect to the Euclidean metric $\omega_{e}$.  In this coordinate, by Lemma \ref{2.5}, we know
\begin{equation*}
\omega(t)|_{\pi^{-1}(y_{0})}\leq \frac{C}{(\delta_{0})^{2}}(\pi_{1})^{*}\omega_{FS}\leq \frac{C'}{(\delta_{0})^{2}}\omega_{e}.
\end{equation*}
Since closed curve $\gamma\subset F$,
\begin{equation}\label{3.4}
\int_{\gamma}\omega(t)=\int_{\gamma}[\frac{1}{T}((T-t)\omega_{0}+t\pi^{*}\omega_{Y})]=\frac{T-t}{T}\int_{\gamma}\omega_{0}\leq C(T-t).
\end{equation}
Write $\sigma=(T-t)^{1/3}>0$, which we may assume is sufficiently small.

Moreover, we may assume that $p$ is represented by the origin in $\mathbb{C}=\mathbb{R}^{2}$, that $q$ is represented by the point $(x_{0},0)$ with $0<x_{0}<1$, and that the rectangle
\begin{equation*}
\mathcal{R}=\{(x,y)\in\mathbb{R}^{2}|0\leq x\leq x_{0}, -\sigma\leq y\leq \sigma\}\subset \mathbb{R}^{2}=\mathbb{C}
\end{equation*}
is contained in the image of $U$. Now in $\mathcal{R}$, the fixed metric $\hat{g}_{0}$ induced from the metric $g_{0}$ on $X$ is uniformly equivalent to the Euclidean metric. Thus from (\ref{3.4}),
\begin{equation}
\int_{-\sigma}^{\sigma}(\int_{0}^{x_{0}}({\rm tr}_{\hat{g}_{0}}g)dx)dy=\int_{\mathcal{R}}({\rm tr}_{\hat{g}_{0}}g)dxdy\leq C(T-t).
\end{equation}
Hence there exists $y'\in (-\sigma, \sigma)$ such that
\begin{equation}
\int_{0}^{x_{0}}({\rm tr}_{\hat{g}_{0}}g)(x,y')dx\leq \frac{C}{\sigma}(T-t)=C(T-t)^{2/3}.
\end{equation}
Now let $p''$ and $q''$ be the points represented by coordinates $(0,y')$ and $(x_{0}, y')$. Then, considering the horizontal path $s\mapsto (s,y')$ between $p''$ and $q''$, we have
\begin{eqnarray}
\begin{aligned}
d_{\omega}(p'',q'')&\leq \int_{0}^{x_{0}}(\sqrt{g(\partial_{x},\partial_{x})})(x,y')dx\\
&=\int_{0}^{x_{0}}(\sqrt{{\rm tr}_{\hat{g}_{0}}g}\sqrt{\hat{g}_{0}(\partial_{x}, \partial_{x})})(x,y')dx\\
&\leq (\int_{0}^{x_{0}}({\rm tr}_{\hat{g}_{0}}g)(x,y')dx)^{1/2}(\int_{0}^{x_{0}}(\hat{g}_{0}(\partial_{x},\partial_{x}))(x,y')dx)^{1/2}\\
&\leq C(T-t)^{1/3}.
\end{aligned}
\end{eqnarray}
and
\begin{equation}
d_{\omega}(p,p'')\leq d_{\omega|_{\pi^{-1}(y_{0})}}(p,p'')\leq \frac{C}{(\delta_{0})^{2}}\sigma=C(T-t)^{1/15}.
\end{equation}
Using the same argument we can prove
\begin{equation}
d_{\omega}(q,q'')\leq C(T-t)^{1/15}.
\end{equation}
Hence by triangle inequality
\begin{equation}
d_{\omega(t)}(p,q)\leq C(T-t)^{1/15}.
\end{equation}
\end{proof}

\subsection{Gromov-Hausdorff Convergence}
In this subsection, we prove that there exists a sequence of metrics along the K\"ahler-Ricci flow converges sub-sequentially to a metric on $Y$ in the Gromov-Hausdorff sense as $t\rightarrow T$.
\begin{lemma}\label{LimitDistance}
Write $d_{t}: X\times X\rightarrow \mathbb{R}$ for the distance function induced by the metric $\omega(t)$. There exists a sequence of times $t_{i}\rightarrow T$, such that the functions $d_{t_{i}}$ converge uniformly to a function $d_{\infty}:X\times X\rightarrow \mathbb{R}$.

Moreover, if, for $p, q\in Y$, we let $d_{Y,\infty}(p,q)=d_{\infty}(\tilde{p},\tilde{q}),$ where $\tilde{p}\in \pi^{-1}(p)$ and $\tilde{q}\in \pi^{-1}(q)$, then $d_{Y,\infty}$ defines a distance function on $Y$, which is uniformly equivalent to that induced by $\omega_{Y}$.
\end{lemma}
\begin{proof}
First note that the functions $d_{t}:X\times X\rightarrow \mathbb{R}$ are uniformly bounded. Indeed by (1) in Theorem \ref{MainThm1}, we have a constant $C>0$ such that $d_{t}(x,y)<C$ for any $t<T$ and $x,y\in X$. Next, we prove that the functions $d_{t}:X\times X\rightarrow \mathbb{R}$ are equicontinuous with respect to the metric on $X\times X$ induced by $d_{0}$.

For any $x,x',y,y'\in X$, we have
\begin{eqnarray}
\begin{aligned}
|d_{t}(x,y)-d_{t}(x',y')|&\leq |d_{t}(x,y)-d_{t}(x,y')|+|d_{t}(x,y')-d_{t}(x',y')|\\
&\leq d_{t}(y,y')+d_{t}(x,x').
\end{aligned}
\end{eqnarray}

Define $|s|_{h}^{2}=hs\overline{s}$, and we assume $|s|_{h}(y)=1$.

Since $Y$ is compact, there exist finite Zariski open sets $U^{1},\cdots, U^{N}$ and biholomorphic map $\Phi_{1},\cdots, \Phi_{N}$ such that
the diagram
$$\xymatrix{
    \pi^{-1}(U^{i}) \ar[rr]^{\Phi_{i}}\ar[dr]_{\pi} & & U^{i}\times F \ar[dl]^{{\rm Pr}_{1}} \\
     & U^{i} &
    }$$
commutes, where $D_{i}=Y\backslash U^{i}$, and $\cup_{i=1}^{N}\pi^{-1}(U^{i}_{1/3})=X$. Let $s_{i}$ be the holomorphic sections on $[D_{i}]$ and let $h_{i}$ be the Hermitian metrics on $[D_{i}]$.  Here $U^{i}_{r}=\{\tilde{y}\in U^{i}||s_{i}|_{h_{i}}>r\}$ for $0<r\leq 1$.

\begin{claim}
There exists a uniform constant $\delta>0$ such that if $x\in \pi^{-1}(U^{i_{0}}_{1/3})$ for some $i_{0}\in\{1,\cdots, N\}$ and $d_{0}(x,x')<\delta$, then $x'\in \pi^{-1}(U^{i_{0}}_{2/3})$.
\end{claim}
\begin{proof}[Proof of Claim]

Denote $A^{i}_{r}$ be the boundary of $\pi^{-1}(U^{i}_{r})$ and denote $d_{0}(A^{i}_{1/3}, A^{i}_{2/3})=\delta_{i}>0$. Let $\rho^{i}_{1}, \rho^{i}_{2}$ be the partition of unity for the cover $(X\backslash \pi^{-1}(U^{i}_{2/3}), \pi^{-1}(U^{i})$. Then $\omega^{i}=\rho^{i}_{1}\omega_{0}+\rho^{i}_{2}\Phi^{*}_{i}({\rm Pr}_{1}^{*}\omega_{Y}+{\rm Pr}_{2}^{*}\omega_{F})$ are the Hermitian metrics on $X$, which are equivalent to metric $\omega_{0}$, and $\omega^{i}=\Phi^{*}_{i}({\rm Pr}_{1}^{*}\omega_{Y}+{\rm Pr}_{2}^{*}\omega_{F})$ on $\pi^{-1}(U^{i}_{2/3})$. Hence there exists a uniform constant $C>0$ such that
\begin{equation}
C^{-1}\omega_{0}\leq \omega^{i}\leq C\omega_{0}.
\end{equation}
We take $\delta=C^{-2}\min\{\delta_{1}, \cdots, \delta_{N}\}.$ Now we can prove $x'\in \pi^{-1}(U^{i_{0}}_{2/3})$. If not, we know
\begin{equation}
d_{\omega^{i_{0}}}(x,x')\geq d_{\omega^{i_{0}}}(A^{i_{0}}_{1/3}, A^{i_{0}}_{2/3})\geq C^{-1}d_{0}(A^{i_{0}}_{1/3}, A^{i_{0}}_{2/3})=C^{-1}\delta_{i_{0}}.
\end{equation}
On the other hand,
\begin{equation}
d_{\omega^{i_{0}}}(x,x')\leq C d_{0}(x,x')<C\delta.
\end{equation}
It is a contradiction.  Hence $x'\in \pi^{-1}(U^{i_{0}}_{2/3})$. We finish the proof of the claim.
\end{proof}
Now if $x,x'\in X$ satisfying $d_{0}(x,x')<\delta$, by the above claim we have $x,x'\in U^{i_{0}}_{2/3}$ for some $i_{0}\in\{1,\cdots,N\}$. Now we choose $q_{0}\in F$ such that $|s|^{2}_{h}(q_{0})=1$. Then by the Lemma \ref{2.5}, we have
\begin{equation}
d_{t}(\Phi_{i_{0}}^{-1}(\pi(x), q_{0}), \Phi_{i_{0}}^{-1}(\pi(x'),q_{0}))\leq Cd_{\omega_{Y}}(\pi(x),\pi(x'))\leq C' d_{0}(x,x').
\end{equation}
By Lemma \ref{EstimateFiber}
\begin{eqnarray}\label{Equi}
\begin{aligned}
d_{t}(x,x')&\leq d_{t}(x, \Phi_{i_{0}}^{-1}(\pi(x), q_{0}))+d_{t}(\Phi_{i_{0}}^{-1}(\pi(x'),q_{0}),x')+d_{t}(\Phi_{i_{0}}^{-1}(\pi(x), q_{0}), \Phi_{i_{0}}^{-1}(\pi(x'),q_{0}))\\
&\leq C[(T-t)^{1/15}+d_{0}(x,x')].
\end{aligned}
\end{eqnarray}
Now we  prove the following lemma.
\begin{lemma}\label{ArzelaAscoli}
With the assumption of $(\ref{Equi})$, there exists a sequence of times $t_{i}\rightarrow T$, such that the functions $d_{t_{i}}$ converges uniformly to a continuous function $d_{\infty}$.
\end{lemma}
\begin{proof}[Proof of Lemma \ref{ArzelaAscoli}]

We denote $M=X\times X$, is a compact manifold. The first thing to recall is
that any compact metric space has a countable dense subset. This follows directly
from the definition of compactness. Namely, given any $k\in\mathbb{N}$, cover $M$ by all the
balls of radius $\frac{1}{k}$ (centred at all the points of $M$.) By compactness of $M$ this has a finite
subcover, let $Q_{k}\subset M$ be the set of centers of such a finite subcover. Then every
point of $M$ is in one of the balls, so it is distant at most $\frac{1}{k}$ from (at least) one of
the points in $Q_{k}$. The union, $Q=\cup_{k}Q_{k}$, of these finite sets is (at most) countable and is
clearly dense in $M$, i.e., any point in $M$ is the limit of a sequence in $Q$.

Let $\{d_{t_{n}}\}$ be a sequence of $d_{t}$ ($t_{n}\rightarrow T$ as $n\rightarrow\infty$). Take a point $q_{1}\in Q$, then $\{d_{t_{n}}(q_{1})\}$ is a bounded sequence in $\mathbb{R}$.
So, by Heine-Borel Theorem, we may extract a subsequence of $d_{t_{n}}$ so that $\{d_{t_{n_{1,j}}}(q_{1})\}$ converges
in $\mathbb{R}$. Since $Q$ is countable we can construct successive subsequences, $d_{t_{n_{k,j}}}$ of the
preceding subsequence $d_{t_{n_{k-1,j}}}$ , so that the $k$th subsequence converges at the first $k$th
point $\{q_{1}, \cdots, q_{k}\}$ of $Q$. Now, the diagonal sequence $d_{t_{n_{i}}}=d_{t_{n_{i,i}}}$ is a subsequence of
$d_{t_{n}}$. So along this subsequence $d_{t_{n_{i}}}(q)$ converges for each point in $Q$. It is a subsequence of the
original sequence $\{d_{t_{n}}\}$, we just denote it as $\{d_{t_{n}}\}$ and we want to show that it converges uniformly; it suffices
to show that it is uniformly Cauchy.

For any given $\epsilon>0$. By the assumption of $(\ref{Equi})$, we can choose $T_{\epsilon}\in [0, T)$ such that for any $t\in [T_{\epsilon}, T)$ we have $C(T-t)^{1/15}\leq \epsilon/6$ and choose $\delta=\epsilon/6C>0$ so that $|d_{t_{n}}(x)-d_{t_{n}}(y)|<\epsilon/3$ whenever $d_{0}(x, y)<\delta$ and $t_{n}\in [T_{\epsilon}, T)$. Next choose $k>1/\delta$.
Since there are only finitely many points in $Q_{k}$ we may choose $N$ so large that for any $n>N$ we have $t_{n}\in[T_{\epsilon}, T)$, then $|d_{t_{n}}(q)-d_{t_{m}}(q)|<\epsilon/3$ if $q\in Q_{k}$ and $n, m>N$. Then for a general point $x\in M$
there exists $q\in Q_{k}$ with $d_{0}(x, q) <1/k< \delta$, so
\begin{equation}
|d_{t_{n}}(x)-d_{t_{m}}(x)|\leq |d_{t_{n}}(x)-d_{t_{n}}(q)| + |d_{t_{n}}(q)-d_{t_{m}}(q)| + |d_{t_{m}}(q)-d_{t_{m}}(x)|<\epsilon
\end{equation}
whenever $n, m>N$. Thus the sequence is uniformly Cauchy, hence uniformly convergent. Hence the function $d_{\infty}$ is continuous. We finish the proof of Lemma \ref{ArzelaAscoli}.
\end{proof}
It follows that $d_{\infty}$ is nonnegative, symmetric and satisfies the triangle inequality.

Let $d_{Y}: Y\times  Y\rightarrow \mathbb{R}$ be the distance function on $Y$ induced by the metric $\omega_{Y}$. From Lemma \ref{LowerBound}, we have a constant $c>0$ such that $d_{t}(x,y)\geq \sqrt{c}d_{Y}(\pi(x), \pi(y))$. It follows that the limit $d_{\infty}$ satisfies
\begin{equation}
d_{\infty}\geq \sqrt{c}d_{Y}(\pi(x),\pi(y)).
\end{equation}
Now we want to prove the upper bound.  When $\pi(x),\pi(y)\in U^{i_{0}}_{1/3}$ for some $i_{0}\in\{1,\cdots,N\}$, by the inequality (\ref{Equi}) and Lemma \ref{2.7}, we have
\begin{equation}
d_{t}(x,y)\leq C(T-t)^{1/15}+Cd_{Y}(\pi(x),\pi(y)).
\end{equation}
This implies that
\begin{equation}\label{LowerBoundDistance}
d_{\infty}(x,y)\leq Cd_{Y}(\pi(x),\pi(y)).
\end{equation}
Now we consider the general case. Assume $\pi(x)\in U^{1}_{1/3}$ and $\pi(y)\in U^{i_{0}}_{1/3}$ for some $i_{0}\neq 1$. Since $\cup_{i=1}^{N}U^{i}_{1/3}=Y$, we know for any minimal geodesic $\gamma(t)$ connecting $\pi(x)$ and $\pi(y)$ with metric $\omega_{Y}$, there exist a finite points $y_{0}=\pi(x), y_{1}=\gamma(t_{1}), y_{2}=\gamma(t_{2}), \cdots, y_{L}=\gamma(t_{L}), y_{L+1}=\pi(y)\in Y$ such that $y_{i}$ and $y_{i+1}$ are in the same $U^{j_{0}}_{1/3}$ for some $j_{0}\in\{1, 2,\cdots, N\}$.  We choose $x_{0}=x, x_{L+1}=y$ and any $x_{1}, \cdots, x_{L}\in X$ satisfying $\pi(x_{i})=y_{i}$. Hence by
\begin{eqnarray}\label{UpperBoundDistance}
\begin{aligned}
d_{Y}(\pi(x),\pi(y))&=\sum_{i=0}^{L}d_{Y}(y_{i},y_{i+1})\\
&\geq C^{-1}\sum_{i=0}^{L}d_{\infty}(x_{i}, x_{i+1})\\
&\geq C^{-1}d_{\infty}(x,y).
\end{aligned}
\end{eqnarray}
For $p, q\in Y$, we now define $d_{Y,\infty}(p,q)=d_{\infty}(\tilde{p},\tilde{q})$, where $\tilde{p}\in\pi^{-1}(p), \tilde{q}\in\pi^{-1}(q)$. This is independent of the choice of lifts $\tilde{p}$ and $\tilde{q}$ since if say $\tilde{p}'$ is a different lift of $p$, then by (\ref{UpperBoundDistance}) and by the triangle inequality, we have
\begin{equation}
d_{\infty}(\tilde{p}', \tilde{q})\leq d_{\infty}(\tilde{p},\tilde{q})+d_{\infty}(\tilde{p}', \tilde{p})=d_{\infty}(\tilde{p},\tilde{q})
\end{equation}
and by switching $\tilde{p}$ and $\tilde{p}'$ we get the reverse inequality. Moreover, it follows from (\ref{LowerBoundDistance}) and (\ref{UpperBoundDistance}) that $d_{Y, \infty}$ is uniformly equivalent to $d_{Y}$. Hence we finish the proof of Lemma \ref{LimitDistance}.
\end{proof}

\begin{theorem}
In the notation of Lemma \ref{LimitDistance} we have $(X,d_{t_{i}})\rightarrow (Y, d_{Y,\infty})$ in the Gromov-Hausdorff sense, where we recall that $d_{t_{i}}$ is the distance function induced by the metric $\omega(t_{i})$.
\end{theorem}
\begin{proof}
Using the same argument in the proof of Theorem 3.1 in \cite{SSW}. Hence finish the proof of (2) in Theorem \ref{MainThm1a}.
\end{proof}

\section{\textbf{$F$ Is Some Weighted Projective Space Blown Up At The Orbifold Point}}
In this section, we will consider the case of the fiber $F$ is the family of m-folds $M_{m,k} (1\leq k<m)$,  was introduced by Calabi \cite{Calabi}, which as generalization of the Hirzebruch surfaces. $M_{n,k}$ is a compactification of the blow up of a $\mathbb{Z}_{k}$-orbifold point of the orbifold $Y_{m,k}$, is a $\mathbb{P}^{1}$-bundle over $\mathbb{P}^{m-1}$.  The detail of the construction of $M_{m,k}$ and $Y_{m,k}$,  please see
\cite{SW2}.

First, we recall the definitions of $M_{m,k}$ and $Y_{m,k}$.

We define $M_{m,k}$ to be the $\mathbb{P}^{1}$-bundle
\begin{equation}
M_{m,k}=\mathbb{P}(\mathcal{O}(-k)\oplus \mathcal{O})
\end{equation}
over $\mathbb{P}^{n-1}$. We will assume in this paper that $m\geq 2$ and $1\leq k<m$ (the latter implies that $M_{m,k}$ is a Fano manifold). Denote by $E_{0}$ and $E_{\infty}$ the divisors in $M_{m,k}$ corresponding to sections of $\mathcal{O}(-k)\oplus \mathcal{O}$ with zero $\mathcal{O}(-k)$ and $\mathcal{O}$ component, respectively (the detail see Section 9 in \cite{SW2}). $E_{0}$ is an exceptional divisor with normal bundle $\mathcal{O}(-k)$ of the type discussed above. The complex manifold $M_{m,k}$ can be described as
\begin{eqnarray}
\begin{aligned}
M_{m,k}=&\{([Z_{1},\cdots,Z_{m}], (\sigma, \mu)\in \mathbb{P}^{m-1}\times ((\mathbb{C}^{m}\times \mathbb{C})\backslash \{(0,0)\})|\sigma  \ {\rm is \ in}\\
&{\rm the\ line}\ \lambda \mapsto (\lambda(Z_{1})^{k},\cdots, \lambda (Z_{m})^{k})\}/\sim,
\end{aligned}
\end{eqnarray}
where
\begin{equation}
([Z_{1},\cdots, Z_{m}], (\sigma, \mu))\sim ([Z_{1},\cdots, Z_{m}], (\sigma', \mu'))
\end{equation}
if there exists $a\in \mathbb{C}^{*}$ such that $(\sigma, \mu)=(a\sigma', a\mu').$ Then $E_{0}$ and $E_{\infty}$ are the divisors $\{\sigma=0\}$ and $\{\mu=0\}$, respectively.

The orbifold $Y_{m,k}$ is the weighted projective space
\begin{equation}
Y_{m,k}=\{(Z_{0}, \cdots, Z_{m})\in \mathbb{C}^{n+1}\}/\sim.
\end{equation}
where $(Z_{0}', \cdots, Z_{m}')\sim (Z_{0}, \cdots, Z_{m})$ if there exists $\lambda \in \mathbb{C}^{*}$ such that
\begin{equation}
(Z_{0}', \cdots, Z_{m}')=(\lambda^{k}Z_{0},\lambda Z_{1}, \cdots, \lambda Z_{m}).
\end{equation}
We write elements of $Y_{m,k}$ as $[Z_{0}, \cdots, Z_{m}]$. Then $Y_{m,k}$ has a single $\mathbb{Z}_{k}$-orbifold point at $[1,0,\cdots,0]$.

We define the map $\pi_{1}: M_{m,k}\rightarrow Y_{m,k}$ by
\begin{equation}
\pi_{1}(([Z_{1},\cdots, Z_{m}], (\sigma, \mu)))=[\mu, b^{1/k}Z_{1}, \cdots, b^{1/k}Z_{m}],
\end{equation}
where $b\in \mathbb{C}$ is defined by
\begin{equation}
\sigma=b((Z_{1})^{k},\cdots, (Z_{m})^{k}).
\end{equation}
$\pi_{1}$ is globally well-defined, surjective and injective on the complement of $E_{0}$.

If we identify the line bundle $\mathcal{O}(-k)$ with the open subset $\{\mu\neq 0\}$ of $M_{m,k}$ and $\mathbb{C}^{m}/\mathbb{Z}_{k}$ with the open subset $\{Z_{0}\neq 0\}$ of $Y_{m,k}$ via $(z_{1}, \cdots, z_{m})\mapsto [1,z_{1},\cdots, z_{m}]$, then $\pi_{1}$ restricted to $M_{m,k}\backslash E_{0}$ is a biholomorphism onto $Y_{m,k}\backslash \{[1,0,\cdots, 0]\}$ and $\pi_{1}(E_{0})=[1,0,\cdots, 0]$.

All of the manifolds $M_{m,k}$ admit K\"ahler metrics. Indeed, the cohomology classes of the line bundles $[E_{0}]$ and $[E_{\infty}]$ span $H^{1,1}(M_{m,k}; \mathbb{R})$ and every K\"ahler class $\alpha$ can be written uniquely as
\begin{equation}
\alpha=\frac{b}{k}[E_{\infty}]-\frac{a}{k}[E_{0}]
\end{equation}
for constants $a,b$ with $0<a<b$. The first chern class
\begin{equation}
c_{1}(M_{m,k})=\frac{m+k}{k}[E_{\infty}]-\frac{n-k}{k}[E_{0}].
\end{equation}
Hence if $1\leq k\leq m-1$, then $M_{m,k}$ is a Fano manifold. He and Sun proved that any weighted projective space exists an orbifold K\"ahler (in fact is K\"ahler-Ricci soliton) metric $\omega_{{\rm orb}}$ with positive bisectional curvature, see Theorem 1.2 in \cite{HS}.

Let $L$ be the $\mathcal(-k)$ line bundle over $\mathbb{P}^{m-1}$, for $k\geq 1$. We give a description of the total space of $L$ as follows. Writing $[Z_{1},\cdots, Z_{m}]$ for the homogeneous coordinates on $\mathbb{P}^{m-1}$, we define
\begin{equation}
L=\{([Z_{1},\cdots, Z_{m}],\sigma)\in \mathbb{P}^{m-1}\times \mathbb{C}^{m}|\sigma \ {\rm is}\ {\rm in} \ {\rm the} \ {\rm line} \ \lambda\mapsto (\lambda(Z_{1})^{k}, \cdots, \lambda(Z_{m})^{k})\},
\end{equation}
and let $P: L\rightarrow \mathbb{P}^{m-1}$ be the projection onto the first factor. Each fiber $P^{-1}([Z_{1},\cdots, Z_{m}])$ is the line in $\mathbb{C}$. $L$ can be given $m$ complex coordinate charts
\begin{center}
$U_{i}=\{([Z_{1}, \cdots, Z_{m}], \sigma)\in L | Z_{i}\neq 0\}$ for $i=1, \cdots, m.$
\end{center}
On $U_{i}$, we have coordinates $\omega_{(i)}^{j}$ for $j=1,\cdots, m$ with $j\neq i$ and $y_{(i)}$. The $\omega_{(i)}^{j}$ are defined by
\begin{center}
$\omega_{(i)}^{j}=Z_{j}/Z_{i}$ for $j\neq i$,
\end{center}
and $y_{(i)}$ by
$$\sigma=\frac{y_{(i)}}{(Z_{i})^{k}}((Z_{1})^{k}, \cdots, (Z_{m})^{k}).$$
On $U_{i}\cap U_{l}$ with $i\neq l$, we have
\begin{center}
$\omega_{(i)}^{j}=\frac{\omega_{(l)}^{j}}{\omega_{(l)}^{i}}$ for $j\neq i,l, \quad \omega_{(i)}^{l}=\frac{1}{\omega_{(l)}^{i}}\ $ and $\ y_{(i)}=y_{(l)}(\frac{Z_{i}}{Z_{l}})^{k}=y_{(l)}(\omega_{(l)}^{i})^{k}.$
\end{center}
Now let $E$ be the submanifold of $L$ defined by the zero section of $L$ over $\mathbb{P}^{m-1}$. Denote by $[E]$ the pull-back line bundle $P^{*}L$ over $L$, which corresponds to the hypersurface $E$. Writing the transition functions of $[E]$ in $U_{i}\cap U_{l}$ as $t_{il}=(Z_{i}/Z_{l})^{k}=y_{(i)}/y_{(l)},$ we have a section $\tilde{s}$ over $[E]$ given by
$$s_{i}: \ U_{i}\rightarrow \mathbb{C}, \quad s_{i}=y_{i}.$$
We can define a Hermitian metric $\tilde{h}$ on the fibers of $[E]$ by
\begin{center}
$h_{i}=\frac{(\Sigma_{j=1}^{m}|Z_{j}|^{2})^{k}}{|Z_{i}|^{2k}}$ on $U_{i}$.
\end{center}
Namely, $\tilde{h}$ is the pull-back of $h_{FS}^{-k}$ where $h_{FS}$ is the Fubini-Study metric on $\mathcal{O}(1)$. We have
\begin{equation}
|\tilde{s}|^{2}_{\tilde{h}}=|y_{(i)}|^{2}\frac{(\Sigma_{j=1}^{m}|Z_{j}|^{2})^{k}}{|Z_{i}|^{2k}} \quad {\rm on} \ U_{i}.
\end{equation}

If we denote $r^{2}=\sum_{i=1}^{m}|z^{i}|^{2}$, then we have
\begin{equation}
\pi_{1}^{*}r^{2k}=|\tilde{s}|^{2}_{\tilde{h}}
\end{equation}
on $L$.

Let $\omega_{e}$ be the standard orbifold metric on $\mathbb{C}^{m}/\mathbb{Z}_{k}$, which lifts to the Euclidean metric on $\mathbb{C}^{m}$, we write $\omega_{e}$ as
\begin{equation}
\omega_{e}=\frac{\sqrt{-1}}{2\pi}\Sigma_{i}dz^{i}\wedge d\overline{z^{i}}.
\end{equation}

Denote $\omega_{F}$ be the metric $\omega_{X}$ in Lemma 2.3 in \cite{SW2}, it is a K\"ahler metric on $M_{m,k}$. We will work in a local uniformizing chart around the orbifold point $p\in Y_{n,k}$, which we identify with the unit ball $D_{1}$ in $\mathbb{C}^{m}$.  Then we know that $\omega_{{\rm orb}}$ is uniformly equivalent to the Euclidean metric $\omega_{e}$ on $D_{1}$.  We write $D_{R}$ for the ball in $\mathbb{C}^{m}$ of radius $R>0$.  Then from the section 2 in \cite{SW2}, on $\pi_{1}^{-1}(D_{1}\backslash \{0\})$ we have
\begin{equation}\label{2.19a}
k|\tilde{s}|_{\tilde{h}}^{2(k-1)/k}\pi_{1}^{*}\omega_{e}\leq \omega_{F}\leq \frac{C}{|\tilde{s}|_{\tilde{h}}^{2/k}}\pi_{1}^{*}\omega_{e}
\end{equation}
for some uniform constant $C>0$.  Hence on $\pi_{1}^{-1}(D_{1}\backslash \{0\})$ there exists a uniform constant $C$ such that
\begin{equation}\label{2.20a}
C^{-1}|\tilde{s}|_{\tilde{h}}^{2(k-1)/k}\pi_{1}^{*}\omega_{{\rm orb}}\leq \omega_{F}\leq \frac{C}{|\tilde{s}|_{\tilde{h}}^{2/k}}\pi_{1}^{*}\omega_{{\rm orb}}.
\end{equation}

Now if we denote $\tilde{\omega}=\Phi^{*}({\rm Pr}_{2}^{*}\pi_{1}^{*}\omega_{{\rm orb}}+{\rm Pr}_{1}^{*}\omega_{Y})$, since the bisectional curvature of $\omega_{{\rm orb}}$ is positive, using the same argument of Lemma \ref{2.5}, we obtain

\begin{lemma}\label{2.5a}
There exist uniform constants $C>0$ and $\alpha>0$ such that for $\omega=\omega(t)$ a solution of the K\"ahler-Ricci flow,
\begin{equation}
\omega(t)\leq \frac{C}{|s|_{h}^{2\alpha}|\tilde{s}|_{\tilde{h}}^{2/k}}\tilde{\omega}.
\end{equation}
\end{lemma}

Using the same notations as in Lemma \ref{2.6}, then we have
\begin{lemma}\label{2.6a}
For $\omega=\omega(t)$ a solution of the K\"ahler-Ricci flow. Then there exist uniform constants $C>0$ and $R_{0}=R_{0}(m,k)\in (0,1)$ such that on $D_{R_{0}}\backslash \{0\}$:
\begin{equation}
|\tilde{V}|_{\omega}^{2}\leq Cr^{2k/k+1}.
\end{equation}
Locally, in $D_{R_{0}}\backslash \{0\}$ we have
\begin{equation}
|W|_{g}^{2}\leq \frac{C}{r^{2/(k+1)}},
\end{equation}
for
$$ W=\sum_{i=1}^{m}(\frac{x^{i}}{r}\frac{\partial}{\partial x^{i}}+\frac{y^{i}}{r}\frac{\partial}{\partial y^{i}})$$
the unit length radial vector field with respect to $g_{e}$, where $z^{i}=x^{i}+\sqrt{-1}y^{i}.$
\end{lemma}
\begin{proof}
We using (\ref{2.19a}) and using the same argument in the proof of of Lemma \ref{2.6}, we can obtain the Lemma.
\end{proof}

Using the same argument in the proof of Lemma \ref{2.7}, we can estimate on the lengths of spherical and radial paths in the punctured ball $D_{R_{0}}\backslash \{0\}$.
\begin{lemma}\label{2.7a}
We have
\begin{itemize}
\item[(1)] For any $y\in Y$ and for $0<r<R_{0}$, the diameter of the $2m-1$ sphere $S_{r}$ of radius $r$ in $D_{R_{0}}$ centered at the origin with the metric induced from $\omega|_{\pi^{-1}(y)}$ is uniformly bounded from above, independent of $r$ and $y$.
\item[(2)] For any $z\in D_{R_{0}}\backslash \{0\}$, the length of a radial path $\gamma(\lambda)=\lambda z$ for $\lambda\in (0, 1]$ with respect to $\omega|_{\pi^{-1}(y)}$ is uniformly bounded from above by a uniform constant multiple of $|z|^{k/(k+1)}$.
\end{itemize}
Hence the diameter of $D_{R_{0}}\backslash \{0\}$ with respect to $\omega|_{\pi^{-1}(y)}$ is uniformly bounded from above and
$$ {\rm diam}(\pi^{-1}(y), \omega|_{\pi^{-1}(y)})\leq C.$$
\end{lemma}

Then using the same argument in Section 3, we can prove the diameter of the fiber along the metrics $g(t)$ tend to zero, then we obtain
\begin{theorem}\label{MainThm1b}
Let $(X, Y, \pi, F)$ be a Fano bundle with $F$ is $M_{m,k}(1\leq k<m)$,  $\omega_{Y}$ be a K\"ahler metric on $Y$ and $\omega_{0}$ be a K\"ahler metric on $X$. Assume $\omega(t)$ is a solution of the K\"ahler-Ricci flow (\ref{KRF}) for $t\in [0, T)$ with initial metric $\omega_{0}$ and $[\omega_{0}]-2\pi Tc_{1}(X)=[\pi^{*}\omega_{Y}]$, then we have
\begin{itemize}
\item[(1)] ${\rm diam}(X, \omega(t))\leq C$ for some uniform constant $C>0$;
\item[(2)] There exists a sequence of times $t_{i}\rightarrow T$ and a distance function $d_{Y}$ on $Y$ (which is uniformly equivalent to the distance induced by $\omega_{Y}$, such that $(X,\omega(t_{i}))$ converge to $(Y,d_{Y})$ in the Gromov-Hausdorff sense.
\end{itemize}
\end{theorem}

\section*{Acknowledgements}
This is part of the first author's thesis at Rutgers University. We thank Professor Jian Song for suggesting the problem and many helpful suggestions.  The work is carried out during the second author's visit at Rutgers University. He would like to thank the Department of Mathematics for its hospitality. His research is partially supported by NSFC Grant No.11301017, Research Fund for the Doctoral Program of Higher Education of China and the Fundamental Research Funds for the Central Universities and the Scholarship from China Scholarship Council.


\begin{thebibliography}{}
\bibitem {A} Aubin, T., \'Equations du type Monge-Amp\`ere sur les vari\'et\'es k\"ahl\'eriennes compactes, Bull. Sci. Math., (2) 102(1978), 63-95.
\bibitem {BCHM} Birkar, C., Cascini, P., Hacon, C. and McKernan, J., Existence of minimal models for varieties of log general type, J. Amer. Math. Soc., 23(2010), 405-468.
\bibitem {Calabi} Calabi, E., Extremal K\"ahler metrics, Seminar on Differential Geometry, Annals of Mathematical Studies 102 (Princeton University Press, Princeton, NJ, 1982) 259-290.
\bibitem {Cao} Cao, H., Deformation of K\"ahler metrics to K\"ahler-Einstein metrics on compact K\"ahler manifolds, Invent. Math. 81(1985), 359-372.
\bibitem {Chen-Tian1} Chen, X. X. and Tian, G., Ricci flow on K\"ahler-Einstein surfaces, Invent. Math., 147(2002), 487-544.
\bibitem {Chen-Tian2} Chen, X. X. and Tian, G., Ricci flow on K\"ahler-Einstein manifolds, Duke Math. J. 131 (2006), 17-73.
\bibitem {Chen-Wang} Chen, X. X. and Wang, B., The K\"ahler-Ricci flow on Fano manifolds (I), J. Eur. Math. Soc., 14(2012), 2001-2038.
\bibitem {Chow et.} Chow, B., Chu, S.-C., Glickenstein, D., Guenther, C., Isenberg, J., Ivey, T., Knopf, D., Lu, P., Luo, F. and Ni, L., The Ricci flow: techniques and applications. Part I. Geometric aspects. Mathematical Surveys and Monographs, 135, American Mathematical Society, Province, RI, 2007.
\bibitem {Collins-Tosatti} Collins, T. and Tosatti, V., K\"ahler currents and null loci, Invent. Math., 202(2015), 1167-1198.
\bibitem {FIK} Feldman, M., Ilmanen, T. and Knopf, D., Rotationally symmetric shrinking and expanding gradient K\"ahler-Ricci solitons, J. Differential Geom.,  65 (2003), no. 2, 169-209.
\bibitem {Fong1} Fong, F. T.-H, K\"ahler-Ricci flow on projective bundles over K\"ahler-Einstein manifolds, Trans. Amer. Math. Soc., 366(2014), no. 2, 563-589.
\bibitem {Fong2} Fong, F. T.-H, On the collapsing rate of the K\"ahler-Ricci flow rate of the K\"ahler-Ricci flow with finite-time singularity, J. Geom. Anal., 25(2015), no. 2, 1098-1107.
\bibitem {Fong-Zhang} Fong, F. T.-H. and Zhang, Z., The collapsing rate of the K\"ahler-Ricci flow with regular infinite time singularity, J. Reine  Angew. Math. 703 (2015), 95-113.
\bibitem {GH} Griffiths, P. and Harris, J., Principles of algebraic geometry, Pure Appl. Math., Wiley-Interscience, New York, 1978.
\bibitem {Gill} Gill, M., Collapsing of products along the K\"ahler-Ricci flow, Trans. Amer. Math. Soc., 366 (2014), no.7, 3907-3924.
\bibitem {Guo} Guo, B., On the K\"ahler-Ricci flow on projective manifolds of general type, arXiv:1501.04239.
\bibitem {Guo-Song-Weinkove} Guo, B., Song, J. and Weinkove, B., Geometric convergence of the K\"ahler-Ricci flow on complex surfaces of general type, arXiv:1505.00705, to appear in IMRN.
\bibitem {H} Hamilton, R. S., Three manifolds with positive Ricci curvature, J. Differential Geom., 17 (1982), no.2, 255-306.
\bibitem {HM} Hacon, C. and McKernan, J., Existence of minimal models for varieties of log general type. II, J. Amer. Math. Soc., 23(2010), 469-490.
\bibitem {HS} He, W. and Sun, S., Frankel conjecture and Sasaki geometry, Adv. Math., 291 (2016), 912-960.
\bibitem {Ilmanen-Knopf} Ilmanen, T. and Knopf, D., A lower bound for the diameter of solutions to the Ricci flow with nonzero $H^{1}(M^{n}; \mathbb{R})$, Math. Res. Lett., 10 (2003), no 2-3, 161-168.
\bibitem {Nave-Tian} La Nave, G. and Tian, G., Soliton-type metrics and K\"ahler-Ricci flow on symplectic quotients, J. Reine Angew. Math. 711(2016), 139-166.
\bibitem {P1} Perelman, G., The entropy formula for the Ricci flow and its geometric applications, arXiv: math.DG/0211159.
\bibitem {P2} Perelman, G., Ricci flow with Surgery on Three-manifolds, arXiv: math.DG/0303109.
\bibitem {P3} Perelman, G., Finite Extinction Time for the Solutions to the Ricci flow on Certain Three-Manifolds, arXiv:math.DG/0307245.
\bibitem {Phong-Sturm} Phong, D. H. and Sturm, J., On stability and the convergence of the K\"ahler-Ricci flow, J. Diff. Geom., 72(2006),149-168.
\bibitem {PSSW1} Phong, D. H., Song, J., Sturm, J. and Weinkove, B., The K\"ahler-Ricci flow and the $\overline{\partial}$-operator on vector fields, J. Diff. Geom., 81(2009), 631-647.
\bibitem {PSSW2} Phong, D. H., Song, J., Sturm, J. and Weinkove, B., The K\"ahler-Ricci flow with positive bisectional curvature, Invent. Math., 173(2008), 651-665.
\bibitem {PSSW3} Phong, D. H., Song, J., Sturm, J. and Weinkove, B., On the convergence of the modified K\"ahler-Ricci flow and solitons, Comment. Math. Helv. 86(2011), 91-112.
\bibitem {Sesum-Tian} Sesum, N. and Tian, G., Bounding scalar curvature and diameter along the K\"ahler-Ricci flow (after Perelman), J. Inst. Math. Jussieu 7 (2008), 575-587.
\bibitem {Song1}  Song, J., Finite time extinction of the K\"ahler-Ricci flow, Math. Res. Lett., 21(2014), no. 6, 1435-1449.
\bibitem {Song2} Song, J., Ricci flow and birational surgery, preprint, arxiv: 1304.2607.
\bibitem {SSW} Song, J., Sz\'ekelyhidi, G. and Weinkove, B., The K\"ahler-Ricci flow on projective bundles, IMRN, 2013, No.2, 243-257.
\bibitem {ST1} Song, J. and Tian, G., The K\"ahler-Ricci flow on surfaces of positive Kodaira dimension, Invent. Math., 170(2007), No.3, 609-653.
\bibitem {ST2} Song, J. and Tian, G., Canonical measures and K\"ahler-Ricci flow, J. Amer. Math. Soc., 25(2012), No.2, 303-353.
\bibitem {ST3} Song, J. and Tian, G., The K\"ahler-Ricci flow through singularities, arxiv: 0909.4898, to appear in Invent. Math.
\bibitem {ST4} Song, J. and Tian, G., Bounding scalar curvature for global solutions of the K\"ahler-Ricci flow, American Journal of Mathematics, 138 (2016), no.2.
\bibitem {SW0} Song, J. and Weinkove, B., The K\"ahler-Ricci flow on Hirzebruch surfaces, J. Reine Angew. Math., 659(2011), 141-168.
\bibitem {SW1} Song, J. and Weinkove, B.,  Contracting exceptional divisors by the K\"ahler-Ricci flow, Duke Math. Jour., 162(2013), No.2, 367-415.
\bibitem {SW2} Song, J. and Weinkove, B., Contracting exceptional divisors by the K\"ahler-Ricci flow II, Proc. London Math. Soc., 108(2014), No. 3, 1529-1561.
\bibitem {SW3} Song, J. and Weinkove, B., Introduction to the K\"ahler-Ricci flow, Chapter 3 of 'Introduction to the K\"ahler-Ricci flow', eds S. Boucksom, P. Eyssidieux, V. Guedj, Lecture Notes Math. 2086, Springer 2013.
\bibitem {Song-Yuan} Song, J. and Yuan, Y., Metric flips with Calabi ansatz, Geom. Funct. Anal., 22(2012), no. 1, 240-265.
\bibitem {Szekely} Szekelyhidi, G., The K\"ahler-Ricci flow and K-polystability, Amer. J. Math., 132 (2010), 1077-1090.
\bibitem {Tian1} Tian, G., New results and problems on K\"ahler-Ricci flow, Ast\'erisque No. 322(2008), 71-92.
\bibitem {Tian2} Tian, G., Finite-time singularity of K\"ahler-Ricci flow, Discrete Contin. Dyn. Syst., 28(2010), no. 3, 1137-1150.
\bibitem {TZZZ} Tian, G., Zhang, S. J., Zhang, Z. L. and Zhu, X. H., Perelman's entropy and K\"ahler-Ricci flow on a Fano manifold, Trans. Amer. Math. Soc. 365(2013), 6669-6695.
\bibitem {Tian-Zhang} Tian, G. and Zhang, Z.,  On the K\"ahler-Ricci flow on projective manifolds of general type, Chinese Ann. Math. Ser. B 27(2006), No. 2, 179-192.
\bibitem {Tian-Zhangzl} Tian, G. and Zhang, Z.-L., Convergence of K\"ahler-Ricci flow on lower dimensional algebraic manifolds of general type, arXiv: 1505.01038.
\bibitem {Tian-Zhangzl1} Tian, G. and Zhang, Z.-L., Regularity of K\"ahler-Ricci flows on Fano manifolds, Acta. Math., 216(2016), 127-176.
\bibitem {Tian-Zhu1} Tian, G. and Zhu, X. H., Convergence of K\"ahler-Ricci flow, J. Amer. Math. Soc., 20(2007) 675-699.
\bibitem {Tian-Zhu2} Tian, G. and Zhu, X. H., Convergence of the K\"ahler-Ricci flow on Fano manifolds, J. Reine Angew. Math. 678(2013), 223-245.
\bibitem {TWY} Tosatti, V., Weinkove, B. and Yang, X., The K\"ahler-Ricci flow, Ricci flat metrics and collapsing limits, arXiv: 1408.0161.
\bibitem {Tosatti-Zhang1} Tosatti, V. and Zhang, Y., Infinite time singularities of the K\"ahler-Ricci flow, Geom. Topol., 19(2015), 2925-2948.
\bibitem {Tosatti-Zhang2} Tosatti, V. and Zhang, Y., Finite time collapsing of the K\"ahler-Ricci flow on threefolds, arXiv: 1507.08397.
\bibitem {Tsuji} Tsuji, H., Existence  and degeneration of K\"ahler-Einstein metrics on minimal algebraic varieties of general type, Math. Ann., 281 (1988), no. 1, 123-133.
\bibitem {Yau} Yau, S., On the Ricci curvature of a compact K\"ahler manifold and the complex Monge-Amp\`ere equation. I, Comm. Pure Appl. Math. 31 (1978), 339-411.
\bibitem {Zhang1} Zhang, Z., Scalar curvature behavior for finite-time singularity of K\"ahler-Ricci flow, Michigan Math. J., 59(2010), no. 2, 419-433.
\bibitem {Zhang2} Zhang, Z., Ricci lower bound for K\"ahler-Ricci flow, Commun. Contemp. Math., 16(2014), no. 2, 1350053, 11pp.
\end{thebibliography}
\end{document}